\documentclass[english, 10pt,reqno]{amsart}
\usepackage{geometry}            
\usepackage{amssymb,amsmath,amsthm,amsfonts,color}
\usepackage{mathrsfs,dsfont, comment,mathscinet}
\usepackage{graphicx}
\usepackage{epstopdf}
\usepackage{mathtools}
\usepackage{babel}
\usepackage{enumerate,esint}
\usepackage{caption}
\usepackage{subcaption}

\usepackage{indentfirst}
\usepackage{bm}
\usepackage{picinpar}
\usepackage{lipsum}

\usepackage[colorlinks=true, pdfstartview=FitV, linkcolor=blue, citecolor=blue, urlcolor=blue]{hyperref}

\usepackage[toc,page]{appendix}

\usepackage{algorithm2e}
\usepackage{etoolbox}
\usepackage{authblk}
\usepackage{amsaddr}
\usepackage{float}

\graphicspath{ {images/} }
\allowdisplaybreaks 

\mathtoolsset{showonlyrefs} 

\makeatletter

\usepackage{fancyhdr}
 
\makeatletter
\patchcmd{\@maketitle}
  {\ifx\@empty\@dedicatory}
  {\ifx\@empty\@date \else {\vskip3ex \centering\footnotesize\@date\par\vskip1ex}\fi
   \ifx\@empty\@dedicatory}
  {}{}
\patchcmd{\@maketitle}
  {\ifx\@empty\@date\else \@footnotetext{\@setdate}\fi}
  {}{}{}
\makeatother

\usepackage{lineno}
\newcommand*\patchAmsMathEnvironmentForLineno[1]{%
  \expandafter\let\csname old#1\expandafter\endcsname\csname #1\endcsname
  \expandafter\let\csname oldend#1\expandafter\endcsname\csname end#1\endcsname
  \renewenvironment{#1}%
     {\linenomath\csname old#1\endcsname}%
     {\csname oldend#1\endcsname\endlinenomath}}%
\newcommand*\patchBothAmsMathEnvironmentsForLineno[1]{%
  \patchAmsMathEnvironmentForLineno{#1}%
  \patchAmsMathEnvironmentForLineno{#1*}}%
\AtBeginDocument{%
\patchBothAmsMathEnvironmentsForLineno{equation}%
\patchBothAmsMathEnvironmentsForLineno{align}%
\patchBothAmsMathEnvironmentsForLineno{flalign}%
\patchBothAmsMathEnvironmentsForLineno{alignat}%
\patchBothAmsMathEnvironmentsForLineno{gather}%
\patchBothAmsMathEnvironmentsForLineno{multline}%
}

\begingroup
\newtheorem{theorem}{Theorem}[section]
\newtheorem{lemma}[theorem]{Lemma}

\endgroup

\theoremstyle{definition}
\begingroup
\newtheorem{define}[theorem]{Definition}
\newtheorem{remark}[theorem]{Remark}

\newtheorem{proposition}[theorem]{Proposition}

\newtheorem{assumption}[theorem]{Assumption}

\endgroup

\newcommand\ba[1]{\begin{align}\label{#1}}
\newcommand\ea{\end{align}}
\newcommand\bas{\begin{align*}}
\newcommand\eas{\end{align*}}

\newcommand\ee{\end{equation}}
\newcommand\be{\begin{equation}}
\newcommand\ees{\end{equation*}}
\newcommand\bes{\begin{equation*}}

\mathchardef\emptyset="001F

\newcommand{\e}{\varepsilon}

\newcommand{\om}{\omega}
 
\newcommand{\R}{{\mathbb R}}

\newcommand{\rn}{{{\R}^N}}
\newcommand{\rnt}{{{\R}^2}}

\newcommand{\wto}{\rightharpoonup}
\newcommand{\wtos}{\mathrel{\mathop{\rightharpoonup}\limits^*}}

\newcommand{\CC}{{\mathcal C}}
\newcommand{\N}{{\mathbb{N}}}

\newcommand\pdeor{\mathscr{A}}

\newcommand\norm[1]{\left\|#1\right\|}

\newcommand{\abs}[1]{\left\lvert#1\right\rvert} 
\newcommand{\fsp}[1]{\left(#1\right)} 
\newcommand{\fmp}[1]{\left[#1\right]}
\newcommand{\flp}[1]{\left\{#1\right\}}

\newcommand{\vp}{\varphi}

\newcommand{\limn}{\lim_{n\rightarrow\infty}}

\newcommand{\B}{{\mathscr B}}
\newcommand{\divg}{{\operatorname{div}}}
\newcommand{\loc}{{\operatorname{loc}}}

\newcommand{\seqn}[1]{\left\{#1\right\}_{n=1}^\infty}  
\newcommand{\seqk}[1]{\left\{#1\right\}_{k=1}^\infty}

\newcommand{\FR}{{\mathcal R}}

\newcommand{\A}{\mathscr{A}}

\newcommand{\PGV}{{PGV^2_{\alpha,\B}}}
\newcommand{\uab}{{u_{\alpha,\B}}}

\newcommand{\PGVon}{{PGV^2_{\B_n}}}

\newcommand{\PGVk}{{PGV^{k+1}_{\alpha,\B[k]}}}

\newcommand{\Pl}{{P_\lambda}}

\newcommand{\mnn}{{\R^{N\times N}}}
\newcommand{\mnnt}{{\mathbb M^{2\times 2}}}
\newcommand{\PGVt}{{PGV_\B^2}}

\newcommand\M{\mathbb M}

\definecolor{CMUred}{RGB}{153,0,0}
\definecolor{CMUgreen}{RGB}{0,135,81}
\definecolor{CMUblue}{RGB}{0,51,127}
\definecolor{Pblue}{RGB}{87,158,208}

\newcommand{\argmin}{{\operatorname{arg\,min}}}

\newcommand{\ta}{{\tilde \alpha}}

\newcommand{\mb}{{\mathcal{M}_b}}
\newcommand{\liminfn}{{\liminf_{n\to\infty}}}
\newcommand{\limsupn}{{\limsup_{n\to\infty}}}

\newcommand{\mbqmnn}{{\mb(Q;\mnn)}}
\newcommand{\mbqmnnt}{{\mb(Q;\mnnt)}}
\newcommand{\mbqrn}{{\mb(Q;\rn)}}

\newcommand{\Aep}{\A}

\def\argmin{\mathop{\rm arg\, min}}

\DeclareRobustCommand{\om}{\omega}

\numberwithin{equation}{section}

\newcommand{\normmm}[1]{{\left\vert\kern-0.25ex\left\vert\kern-0.25ex\left\vert #1 
    \right\vert\kern-0.25ex\right\vert\kern-0.25ex\right\vert}}

\title{Adaptive image processing: first order PDE constraint regularizers and a bilevel training scheme}

\author{Elisa Davoli }
 \address[Elisa Davoli ]{Faculty of Mathematics, University of Vienna\\ Oskar-Morgenstern-Platz 1, A-1090 Vienna, Austria}
 \email[Elisa Davoli ] {elisa.davoli@univie.ac.at}

\author[Irene Fonseca]{Irene Fonseca}
 \address[Irene Fonseca]{Center of Nonlinear Analysis, Department of Mathematics,\\ Carnegie Mellon University, 5000 Forbes Avenue, Pittsburgh, PA, 15213, USA}
 \email[I. Fonseca] {fonseca@andrew.cmu.edu}

\author[P. Liu] {Pan Liu}
 \address[Pan Liu]{Centre of Mathematical Imaging and Healthcare,\\ 
Department of Pure Mathematics and Mathematical Statistics, \\
 University of Cambridge, Wilberforce Road, Cambridge CB3 0WA, UK}
 \email[P. Liu] {panliu.0923@maths.cam.ac.uk}

\subjclass[2010]{26B30, 94A08, 	47J20}
\keywords{image processing, optimal training scheme, first order differential operators, $\Gamma$-convergence}

\date{\today}                                           

\begin{document}
\begin{abstract}
A bilevel training scheme is used to introduce a novel class of regularizers, providing a unified approach to standard regularizers $TV$, $TGV^2$ and $NsTGV^2$. Optimal parameters and regularizers are identified, and the existence of a solution for any given set of training imaging data is proved by $\Gamma$-convergence. Explicit examples and numerical results are given.
\end{abstract}
\maketitle
\tableofcontents

\thispagestyle{empty}

\section{Introduction}

Image processing aims at the reconstruction of an original ``clean" image starting from a ``distorted one", namely from a datum which has been deteriorated or corrupted by noise effects or damaged digital transmission. The key idea of variational formulations in image-processing consists in rephrasing this problem as the minimization of an underlying functional of the form
\be
\mathcal I(u):=\norm{u-u_\eta}_{L^2(Q)}^2+\mathcal R_\alpha(u),
\ee
where $u_\eta$ is a given corrupted image, $Q:=(-1/2,1/2)^N$ is the $N$-dimensional unit square (in image processing we usually take $N=2$, i.e., $Q$ represents the domain of a square image) and $\mathcal R_\alpha$ is a regularizing functional, with $\alpha$ denoting the intensity parameter (which could be a positive scalar or a vector). Minimizing the functional $\mathcal I$ allows to reconstruct a ``clean" image based on the functional properties of the regularizer $\mathcal R_\alpha$.\\\\
Within the context of image {denoising}, for a fixed regularizer $\FR_{\alpha}$ we seek to identify 
\be
u_{\alpha,\FR}:=\argmin\flp{\norm{u-u_\eta}_{L^2(Q)}^2+\mathcal R_\alpha(u):\,\, u\in L^2(Q)}.
\ee
An example is the {ROF} model (\cite{rudin1992nonlinear}), in which the regularizer is taken to be $\FR_\alpha(u):=\alpha TV(u)$, where $TV(u)$ is the {total variation} of $u$ (see, e.g. \cite[Chapter 4]{ambrosio.fusco.pallara}), $\alpha\in\R^+$ is the tuning parameter, and we have
\be\label{intro_denoisy2}
u_{\alpha,TV}:=\argmin\flp{\norm{u-u_\eta}_{L^2(Q)}^2+\alpha TV(u):\,\, u\in L^2(Q)}.
\ee
In view of the coercivity of the minimized functional, the natural class of competitors in \eqref{intro_denoisy2} is $BV(Q)$, the space of real-valued functions of bounded variation in $Q$.
The trade-off between the denoising effects of the {ROF}-functional and its feature-preservation capabilities is encoded by the tuning parameter $\alpha\in \R^+$. Indeed, high values of $\alpha$ lead to a strong penalization of the total variation of $u$, which in turn determines an over-smoothing effect and a resulting loss of information on the internal edges of the reconstructed image, while small values of $\alpha$ cause an unsatisfactory noise removal.\\\\
In order to determine the optimal $\alpha$, say $\ta$, in \cite{reyes2015structure, MR3592840} the authors proposed a {bilevel training scheme}, which was originally introduced in Machine Learning and later adopted by the imaging processing community (see \cite{chen2013revisiting, chen2014insights, domke2012generic,tappen2007learning}). The bilevel training scheme is a semi-supervised training scheme that optimally adapts itself to the given ``clean data". To be precise, let $(u_\eta,u_c)$ be a pair of given images, where $u_\eta$ represents the corrupted version and $u_c$ stands for the original version, or the ``clean" image. This training scheme searches for the optimal $\alpha$ so that the recovered image $u_{\alpha, TV}$, obtained in \eqref{intro_denoisy2}, minimizes the $L^2$-distance from the clean image ${u_c}$. An implementation of such training scheme, denoted by $(\mathcal T)$, equipped with total variation $TV$ is
\begin{flalign}
\text{Level 1. }&\,\,\,\,\,\,\,\,\,\,\,\,\,\,\,\,\,\,\,\,\,\,\,\,\,\,\,\,\,\,\,\,\,\,\,\,\,\ta\in\argmin\flp{\norm{u_{\alpha,TV}-u_c}_{L^2(Q)}^2:\,\,\alpha\in\R^+},&\tag{$\mathcal T$-L1}\label{training_level_bi}\\
\text{Level 2. }&\,\,\,\,\,\,\,\,\,\,\,\,\,\,\,\,\,\,\,\,\,\,\,\,\,\,\,\,\,\,\,\,\,\,\,\,\,u_{\alpha,TV}:=\argmin\flp{\norm{u-u_\eta}_{L^2(Q)}^2+\alpha TV(u):\,\,u\in BV(Q)}.\tag{$\mathcal T$-L2}\label{result_level_bi}&
\end{flalign}
An important observation is that the geometric properties of the regularizer $TV$ play an essential role in the identification of the reconstructed image $u_{\alpha, TV}$ and may lead to a loss of some fine texture in the image. The choice of a given regularizer $\FR_{\alpha}$ is indeed a crucial step in the formulation of the denoising problem: on the one hand, the structure of the regularizer must be such that the removal of undesired noise effects is guaranteed, and on the other hand the disruption of the essential details of the image must be prevented. For this reasons, various choices of regularizers have been proposed in the literature. For example, the second order total generalized variation, $TGV^2_\alpha$,  defined as
\begin{multline}\label{tgv_eq_def}
TGV^2_{\alpha}(u):=\inf\flp{\alpha_0\abs{D u-v}_{\mb(Q;\rn)}+\alpha_1\abs{(\operatorname{sym}\nabla) v}_{\mb(Q;\,\mnn)}:\right.\\
\left.v\in L^1(Q;\rn),\,(\operatorname{sym}\nabla) v\in \mb(Q;\mnn)},
\end{multline}
has been characterized in \cite{bredies2010total}, where $Du$ denotes the distributional gradient of $u$, $(\operatorname{sym}\nabla) v:=(\nabla v+\nabla^Tv)/2$,  $\mb(Q;\mnn)$ is the space of bounded Radon measures in $Q$ with values in $\mnn$, $\alpha_0$ and $\alpha_1$ are positive tuning parameters, and $\alpha:=(\alpha_0,\alpha_1)$. A further commonly used regularizer is the non-symmetric counterpart of the $TGV^2_{\alpha}$- seminorm defined above, namely the $NsTGV^2_{\alpha}$ functional (see e.g., \cite{valkonen}), which is known to provide in general more accurate results compared to $TGV^2_{\alpha}$ but with a higher computational cost. It has been shown that a reconstructed image presents several perks and drawbacks according to the different regularizers. An important question is thus how to identify the regularizer that might provide the best possible image denoising for a given class of corrupted images.\\

To address this problem, it is natural to use a straightforward modification of scheme $(\mathcal T)$ by inserting different regularizers inside the training level 2 in \eqref{result_level_bi}. Namely, we set
\begin{flalign}
\text{Level 1. }&\,\,\,\,\,\,\,\,\,\,\,\,\,\,(\tilde \FR_\alpha):=\argmin\flp{\norm{u_{\alpha,\FR}-u_c}_{L^2(Q)}^2:\,\,\FR_\alpha\in\flp{\alpha TV,TGV^2_{\alpha}, NsTGV_\alpha^2 }},&\,\label{intro_finite_train}\\
\text{Level 2. }&\,\,\,\,\,\,\,\,\,\,\,\,\,\,u_{\alpha,\FR}:=\argmin\flp{\norm{u-u_\eta}_{L^2(Q)}^2+\FR_\alpha(u):\,\,u\in L^1(Q)}.&
\end{flalign}
However, the finite number of possible choices for the regularizer within this training scheme would imply that the optimal regularizer $\tilde \FR_{\alpha}$ would simply be determined by performing scheme $(\mathcal T)$ finitely many times, at each time with a different regularizer $\FR_{\alpha}$. In turn, some possible texture effects for which an ``intermediate" (or interpolated) reconstruction between the one provided by, say, $TGV^2_{\alpha}$ and $NsTGV^{2}_{\alpha}$, might be more accurate, would then be neglected in the optimization procedure. Therefore, one main challenge in the setup of such a training scheme is to give a meaningful interpolation between the regularizers used in \eqref{intro_finite_train}, and also to guarantee that the collection of the corresponding functional spaces exhibits compactness and lower semicontinuity properties.\\\\
The aim of this paper is threefold. First, we propose a novel class of image-processing operators, the PDE-constrained total generalized variation operators, or $PGV^2_{\alpha,\mathscr B}$, defined as
\begin{multline}\label{abextension}
PGV^2_{\alpha,\B}(u):=\inf\flp{\alpha_0\abs{D u-v}_{\mb(Q;\rn)}+\alpha_1\abs{\B v}_{\mb(Q;\,\mnn)}:\right.\\
\left. v\in L^1(Q;\rn),\,\mathscr B v\in \mb(Q;\mnn)},
\end{multline}
for each $u\in L^1(Q;\rn)$, where $\B$ is a linear differential operator (see Section \ref{sec:notation} and Definition \ref{def-tgv-B}) and $\alpha:=(\alpha_0,\alpha_1)$, with $\alpha_0,\,\alpha_1\in (0,+\infty)$. We also define the space of {functions with bounded second order $PGV^2_{\alpha,\B}$-seminorms}
\be
BPGV^2_{\alpha,\B}(Q):=\flp{u\in L^1(Q):\,\,{PGV_{\alpha,\B}^2(u)<+\infty}}.
\ee
Note that if $\B:={\rm sym }\nabla$, then the operator $PGV^2_{\alpha,\B}$ defined in \eqref{abextension} coincides with the operator $TGV^2_\alpha$ mentioned in \eqref{tgv_eq_def}. In fact, we will show that, under appropriate assumptions (see Definition \ref{ready_AB_to_work}), our new operator provides a unified approach to the standard regularizers mentioned in \eqref{intro_finite_train}, generalizing the results in \cite{2018arXiv180201895B} (see Section \ref{uatta_sec}). Moreover, the collection of functionals described in \eqref{abextension} naturally incorporates the recent {PDE}-based approach to image denoising formulated  in \cite{barbu.marinoschi} via nonconvex optimal control problem, thus offering a very general and abstract framework to simultaneously describe a variety of different image-processing techniques.\\

 The second main goal of this article is the study of the training scheme introduced in \eqref{training_level_bi}-\eqref{result_level_bi} that optimizes the trade-off between effective reconstruction and fine image-detail preservation. That is, we propose a new bilevel training scheme that simultaneously yields the optimal regularizer $PGV^2_{\alpha,\B}(u)$ in the class described in \eqref{abextension} and an optimal tuning parameter $\alpha$, so that the corresponding reconstructed image $u_{\alpha,\mathscr B}$, obtained in Level 2 of the $(\mathcal T_\theta^2)$-scheme (see \eqref{ABsolution_map_0_intro} below), minimizes the $L^2$-distance from the original clean image $u_c$. To be precise, in Sections \ref{PABBQ_sec}, \ref{gamma_conv_sec}, and \ref{sec_ts_PGV} we study the improved training scheme $\mathcal T^2_\theta$, for $\theta\in(0,1)$, defined as follows
\begin{flalign}
{\text{Level 1. }}&\,\,\,\,\,\,\,\,\,\,\,\,\,\,\,\,\,\,\,(\tilde \alpha,\tilde\B):= \argmin\flp{\norm{u_c-u_{\alpha,\B}}_{L^2(Q)}^2:\,\,\alpha\in[\theta,1/\theta]^2,\,\,\B\in\Sigma},\tag{$\mathcal T_\theta^2$-L1}\label{ABtraining_0_1_0_intro} &\\
{\text{Level 2. }}&\,\,\,\,\,\,\,\,\,\,\,\,\,\,\,\,\,\,\,u_{\alpha,\B}:=\argmin\flp{\norm{u-u_\eta}_{L^2(Q)}^2+ \PGV(u),\,\, u\in PGV_\B^2(Q)}\tag{$\mathcal T_\theta^2$-L2}\label{ABsolution_map_0_intro},&
\end{flalign}
where $\Sigma$ is an infinite collection of first order linear differential operators $\B$ (see Definition \ref{def_class_Pi} and Definition \ref{def_training_set}). We prove the existence of optimal solutions to \eqref{ABtraining_0_1_0_intro} by showing that the functional 
\be
\mathcal I_{\alpha,\B}(u):=\norm{u-u_\eta}_{L^2}^2+ \PGV(u)
\ee
is continuous in the $L^1$ topology, in the sense of $\Gamma$-convergence, with respect to the parameters $\alpha$ and the operators $\B$ (see Theorem \ref{thm:new-Gamma}). A simplified statement of our main result (see Theorem \ref{main_thm}) is the following.
\begin{theorem}\label{main_thm-intro}
Let $\theta\in(0,1)$ be fixed. Then the training scheme $(\mathcal T_\theta^{2})$ admits at least one solution $(\ta, \tilde \B)\in[\theta,1/\theta]^{2}\times \Sigma$, and provides an associated optimally reconstructed image $u_{\ta,\tilde \B}\in BV(Q)$.
\end{theorem}

 The collection $\Sigma$ of operators $\B$ used in \eqref{ABtraining_0_1_0_intro} has to satisfy several  natural regularity and ellipticity assumptions, which are fulfilled by $\B:=\nabla$ and $\B:=\operatorname{sym}\nabla$ (see Section \ref{sce_better_other}). The general requirements on $\B$ that allow scheme $(\mathcal T^2_\theta)$ to have a solution are listed on Assumptions \ref{assum_basic_B} and \ref{assum_basic_BPi}. Later in Section \ref{fsaqs_sec}, as the third main contribution of this article, we provide in Definition \ref{ready_AB_to_work} a collection of operators $\B$ satisfying Assumptions \ref{assum_basic_B} and \ref{assum_basic_BPi}. In particular, we prove the following (see Theorem \ref{thm_PiA_eq_Pi}).
 
 \begin{theorem}
 \label{thm:main2-intro}
 Let $\mathscr{B}$ be a first order differential operator such that there exists a differential operator $\mathscr{A}$ for which $(\mathscr{A},\mathscr{B})$ is a training operator pair according to the ellipticity assumptions in Definition \ref{ready_AB_to_work}. Then $\mathscr{B}$ satisfies Assumptions \ref{assum_basic_B} and \ref{assum_basic_BPi}. 
 \end{theorem}

Finally, in Section \ref{uatta_sec} we give several explicit examples to show that our class of regularizers $PGV^2_{\alpha,\B}$ includes the seminorms $TGV^2_{\alpha}$ and $NsTGV^2_{\alpha}$, as well as $TV$-variant seminorms and smooth interpolations between them. \\\\
We remark that the task of determining not only the optimal tuning parameter but also the optimal regularizer for given training image data $(u_\eta,u_c)$, has been undertaken in \cite{liu2016weightedreg} where we have introduced one dimensional real order $TGV^r$ regularizers, $r\in[1,+\infty)$, as well as a bilevel training scheme that simultaneously provides the optimal intensity parameters and order of derivation for one-dimensional signals. Forthcoming work in this direction will be found in \cite{liulud2019Image,2018arXiv180506761L}.\\\\
Our analysis is complemented by numerical simulations of the proposed bilevel training scheme. Although this work focuses mainly on the theoretical analysis of the operators $\PGV$ and on showing the existence of optimal results for the training scheme $(\mathcal T^2)$, in Section \ref{sec:num} a primal-dual algorithm for solving \eqref{ABsolution_map_0_intro} is discussed, and some preliminary numerical examples, such as image denoising, are provided.\\\\
With this article we initiate our study of the combination of {PDE}-constraints and bilevel training schemes in image processing. Our follow-up work will include, but is not limited to, the following two topics: 
\begin{itemize}
\item
the construction of a finite grid approximation in which the optimal result $(\tilde\alpha,\tilde\B)$ for the training scheme $(\mathcal T^2_{\theta})$ can be efficiently determined, with an estimation of the approximation accuracy;
\item
 spatially dependent differential operators and multi-layer training schemes. This will allow to specialize the regularization according to the position in the image, providing a more accurate analysis of complex textures and of images alternating areas with finer details with parts having sharpest contours (see also \cite{MR3723325}). 
\end{itemize}

This paper is organized as follows: in Section \ref{sec:notation} we collect some notations and preliminary results. In Section \ref{PABBQ_sec} we analyze the main properties of the $PGV^2_{\alpha,\B}$-seminorms. The $\Gamma$-convergence result and the bilevel training scheme are the subjects of Sections \ref{gamma_conv_sec} and \ref{sec_ts_PGV}, respectively. Section \ref{fsaqs_sec} is devoted to the analysis of the space $BV_{\mathscr{B}}$ for suitable differential operators $\B$. The numerical implementation of some explicit examples is performed in Section \ref{sec:num}.
\section{Notations and preliminary results}\label{sec:notation}
 We collect below some notation that will be adopted in connection with differential operators. Let $N\in \N$ be given, and let $Q:=(-1/2,1/2)^N$ the unit open cube in $\rn$. $\mathbb M^{N^l}$ is the space of real tensors of order $N\times N\times\cdots \times N$ ($l$ times). Note that for $l=1$, the space $\mathbb M^N:=\mathbb M^{N^1}$ is identified with $\rn$, whereas for $l=2$, the collection of second order tensors $\mathbb M^{N^2}$ is identified with the space of $N\times N$ matrices having real entries, usually denoted by $\R^{N\times N}$. For this reason, for $l=2$ and $l=1$ we will directly write $\rn$ and $\R^{N\times N}$ in place of $\mathbb M^N$ and $\mathbb M^{N^2}$, respectively. Also, $\mathcal D'(Q, \mathbb M^{N^l})$ stands for the space of distributions with values in $\mathbb M^{N^l}$, and $\R^N_+$ denotes the set of vectors in $\R^N$ having positive entries.\\

For every open set $U\subset \R^N$, the notation $\B$ will be used for first order differential operators $\B:\mathcal D'(U;\mathbb M^{N^l})\to \mathcal D'(U;\mathbb M^{N^{l+1}})$ defined as 
\be\label{A_quasiconvexity_operator}
\mathscr B v:=\sum_{i=1}^N B^i\frac{\partial}{\partial x_i} v\quad\text{for every }v\in \mathcal D'(U;\mathbb M^{N^l}),
\ee
where $B^i\in \mathbb M^{N^{l+1}}$ for each $i=1,\dots,N$, and where $\frac{\partial}{\partial x_i}$ denotes the distributional derivative with respect to the $i$-th variable. In particular, for $l=1$, there holds $B^i\in \mathbb M^{N^3}$ for each $i=1,\dots,N$, and \eqref{A_quasiconvexity_operator} rewrites as
\be\label{A_quasiconvexity_operator-components}
(\mathscr B v)_{lj}:=\sum_{i,k=1}^N B^i_{ljk}\frac{\partial}{\partial x_i} v_k\quad\text{for every }v\in \mathcal D'(U;\rn),\quad l,j=1,\dots,N.
\ee
For $l=1$, we additionally write the symbol of $\mathscr B$ as 
\be\label{bbB_notation}
\mathbb{B}[\xi]:=\sum_{i=1}^N \xi_i B^i\quad\text{for every }\xi=(\xi_1,\dots,\xi_N)\in \mathbb S^{N-1}.
\ee
Given a sequence $\seqn{\B_n}$ of first order differential operators and a first order differential operator $\B$, with coefficients $\seqn{B^i_n}$ and $B^i$, $i=1,\dots, N$, respectively, we say that $\B_n\to\B$ in $\ell^{\infty}$ if
\be\label{BB_distance}
\norm{\B_n-\B}_{\ell^{\infty}}:=\sum_{i=1}^N\norm{B^i_n-B^i}\to 0,
\ee
where for $B\in \mathbb{M}^{N^l}$, $l\in \N$, $\|B\|$ stands for its Euclidean norm.
\section{The space of functions with bounded $PGV$- seminorm}\label{PABBQ_sec}
\subsection{The space $BV_{\B}$ and the class of admissible operators}
We generalize the standard total variation seminorm by using first order differential operators $\B$: $\mathcal D'(Q;\mathbb M^{N^l})\to \mathcal D'(Q;\mathbb M^{N^{l+1}})$ in the form \eqref{A_quasiconvexity_operator}.
\begin{define}\label{def_BVB}
For every $l\in \N$, we define the space of tensor-valued functions $BV_\B(Q;\mathbb M^{N^l})$ as
\be\label{BVB_norm}
BV_\B(Q;\mathbb M^{N^l}):=\flp{u\in L^1(Q;\mathbb M^{N^l}):\,\, \B u\in \mathcal M_b(Q,\mathbb M^{N^{l+1}})},
\ee
and we equip it with the norm
\be\label{BVB_norm2}
\norm{u}_{BV_\B(Q;\mathbb M^{N^l})}:=\norm{u}_{L^1(Q;\mathbb M^{N^l})}+\abs{\B u}_{\mathcal M_b(Q;\mathbb M^{N^{l+1}})}.
\ee
\end{define}
In order to introduce the class of admissible operators, we first list some assumptions on the operator $\B$.
\begin{assumption}\label{assum_basic_B}
\begin{enumerate}[1.]
\item
The space $BV_\B(Q;\mathbb M^{N^l})$ is a Banach space with respect to the norm defined in \eqref{BVB_norm}.
\item
The space $C^\infty(\overline{Q},\mathbb M^{N^l})$ is dense in $BV_\B(Q;\mathbb M^{N^l})$ in the strict topology. In other words, for every $u\in BV_\B(Q;\mathbb M^{N^l})$ there exists $\seqn{u_n}\subset C^\infty(\bar Q;\mathbb M^{N^l})$ such that 
\be
u_n\to u\text{ strongly in }L^1(Q;\mathbb M^{N^l})\text{ and }\abs{\B u_n}_{\mathcal M_b(Q;\mathbb M^{N^{l+1}})}\to \abs{\B u}_{\mathcal M_b(Q;\mathbb M^{N^{l+1}})}.
\ee
\item
(Compactness) The injection of $BV_\B(Q;\mathbb M^{N^l})$ into $L^1(Q;\mathbb M^{N^l})$ is compact.\\
\end{enumerate}
We point out that, for $l=1$, Requirement 3 above is satisfied for $\B:=\nabla$.
\end{assumption}

The following compactness property applies to a collection of operators $\seqn{\B_n}$.
\begin{assumption}\label{assum_basic_BPi}
Let $\seqn{v_n,\B_n}$ be such that $\B_n$ satisfies Assumption \ref{assum_basic_B} for every $n\in \N$, and
\be\label{eq_assum_basic_BPi}
\sup\flp{\norm{\B_n}_{\ell^{\infty}}+\norm{v_n}_{BV_{\B_n}(Q;\mathbb M^{N^l})}:\,\, n\in\N}<+\infty.
\ee
Then there exist $\B$ and $v\in BV_{\B}(Q;\mathbb M^{N^l})$ such that, up to a subsequence (not relabeled), 
\be\label{eq_v0_unif}
v_n\to v\text{ strongly in }L^1(Q;\mathbb M^{N^l}),
\ee
and 
\be\label{eq_B_n_uniform_bdd2}
{\B_nv_n}\wtos {\B v}\text{ $\text{weakly}^\ast$ in }\mb(Q;\mathbb M^{N^{l+1}}).
\ee
\end{assumption}
\begin{define}\label{def_class_Pi}
For every $l\in \N$, we denote by $\Pi_l$ the collection of operators $\B$ defined in \eqref{A_quasiconvexity_operator}, with finite dimensional null-space $\mathcal N(\B)$, and satisfying Assumption \ref{assum_basic_B}. For simplicity, the class $\Pi_1$ will be indicated by $\Pi$.
\end{define}
In Section \ref{fsaqs_sec} we will exhibit a subclass of operators $\B\in\Pi$ additionally fulfilling the compactness and closure Assumption \ref{assum_basic_BPi}.

\subsection{The $PGV$- total generalized variation}
We introduce below the definition of the PDE-constrained total generalized variation seminorms.
\begin{define}\label{def-tgv-B}
Let $u\in L^1(Q)$ be given. For every $\alpha=(\alpha_0,\alpha_1)\in\R^{2}_+$ and $\B$: $\mathcal D'(Q;\rn)\to \mathcal D'(Q;\mnn)$, $\B\in\Pi$, we consider the seminorm
\be\label{eq:def-tgv-B}
\PGV(u):=\inf\flp{\alpha_0\abs{Du-v}_{\mb(Q;\R^N)}+\alpha_1\abs{\B v}_{\mb(Q;\,\R^{N\times N})}:\,\,v\in BV_\B(Q;\R^N)},
\ee
where the space $BV_\B$ is defined in Definition \eqref{def_BVB}. 
\end{define}
Similarly, we also define the space $\PGVk$ of seminorms of order $k+1\in\N$. We will use the notation $\Pi^k$ to indicate the product $\Pi^k:=\Pi\times\Pi_2\times\cdots\times\Pi_k$.

\begin{define}\label{def-tgv-B_k}
Let $u\in L^1(Q)$ and $k\in \mathbb{N}$ be given. For every $\alpha=(\alpha_0,\ldots,\alpha_k)\in\R^{k+1}_+$ and \mbox{$\B^l$: $L^1(Q;\mathbb M^{N^{l+1}})\to \mathcal D'(Q;\mathbb M^{N^{l+2}})$}, $\B^l\in\Pi_{l+1}$, $l=0,\ldots, k-1$, we consider the seminorm 
\begin{align}
PGV_{\alpha,\B[k]}^{k+1}(u):=&\inf\{\alpha_0\abs{Du-v_0}_{\mb(Q;\R^N)}+\alpha_1\abs{v_1-\B^0 v_0}_{\mb(Q;\,\R^{N\times N})}\label{eq:def-tgv-B_k}\\
&\quad+\alpha_2\abs{v_2-\B^1 v_1}_{\mb(Q;\,\mathbb M^{N^3})}+\cdots+\alpha_{l}\abs{v_{l}-\B^{l-1} v_{l-1}}_{\mb(Q;\,\mathbb M^{N^{l+1}})}\notag\\
&\quad+\cdots+\alpha_{k-1}\abs{v_{k-1}-\B^{k-2} v_{k-2}}_{\mb(Q;\,\mathbb M^{N^k})}\notag\\
&\quad+\alpha_{k}\abs{\B^{k-1} v_{k-1}}_{\mb(Q;\,\mathbb M^{N^{k+1}})}:\,\,v_l\in BV_{\B^l}(Q;\mathbb M^{N^{l+1}}),\,\,l=0,\ldots,k-1\},
\end{align}
where $ \B[k]:=(\B_0,\B_1,\ldots,\B_k)$, and the space $BV_{\B^l}$ is defined in Definition \eqref{def_BVB}. 
\end{define}

We note that for fixed $k\in\N$ and for all $\alpha\in\R^{k+1}_+$, the seminorms $\PGVk$ are topologically equivalent. With a slight abuse of notation, in what follows we will write $PGV^2_{\B}$ instead of $PGV^2_{\alpha,\B}$ and, respectively, $PGV^{k+1}_{\B[k]}$ instead of $PGV^{k+1}_{\alpha,\B[k]}$ whenever the dependence of the seminorm on a specific multi-index $\alpha$ will not be relevant for the presentation of the results. \\

We introduce below the sets of functions with {bounded PDE- generalized variation-seminorms}.

\begin{define}\label{def:BPGV}
We define
\be
BPGV^2_{\B}(Q):=\flp{u\in L^1(Q):\,\, PGV^2_{1,\B}(u)<+\infty},
\ee
and we write
\be
\norm{u}_{BPGV^2_{\B}(Q)}:=\norm{u}_{L^1(Q)}+PGV^2_{1,\B}(u).
\ee

Similarly, we denote by $BPGV^{k+1}_{\B[k]}(Q)$ the space
\be
BPGV^{k+1}_{\B[k]}(Q):=\flp{u\in L^1(Q):\,\, PGV^{k+1}_{1,\B[k]}(u)<+\infty}
\ee
and we write
\be
\norm{u}_{BPGV^{k+1}_{\B[k]}(Q)}:=\norm{u}_{L^1(Q)}+PGV^{k+1}_{1,\B[k]}(u).
\ee
\end{define}
We next show that the $PGV^{k+1}_{\B[k]}$-seminorm and the $TV$-seminorm are equivalent.
\begin{proposition}\label{weak_equi_norm}
 Let $u\in L^1(Q)$ and recall $PGV_{\B[k]}^{k+1}(u)$ from Definition \ref{def-tgv-B_k}. Then for every $k\in \mathbb{N}$, $PGV_{\B[k]}^{k+1}(u)<+\infty$ if and only if $u\in BV(Q)$.
\end{proposition}
\begin{proof}
We notice that by setting $v_l=0$, $l=0,1,\ldots,k-1$, in \eqref{eq:def-tgv-B_k}, we have
\be
\label{eq:ineq-bv}
PGV_{\B[k]}^{k+1}(u)\leq  \abs{Du}_{\mb(Q;\R^N)}
\ee
for every $u\in L^1(Q)$. Thus, if $u\in BV(Q)$ then $PGV_{\B[k]}^{k+1}(u)<+\infty$. \\\\
Conversely, assume that $PGV_{\B[k]}^{k+1}(u)<+\infty$. We only study the case in which $k=1$, as the case $k>1$ can be treated in a completely analogous way. Since $\PGVt(u)<+\infty$, there exists $\bar{v}\in BV_{\B}(Q)$ such that
 \be
 \PGVt(u)\geq \abs{Du-\bar{v}}_{\mb(Q;\R^N)}+\abs{\B \bar{v}}_{\mb(Q;\mnn)}-1.
 \ee
 It suffices to observe that
 \be
 \abs{Du}_{\mb(Q;\R^N)}\leq \abs{Du-\bar{v}}_{\mb(Q;\R^N)}+\|\bar{v}\|_{L^1(Q;\R^N)}\leq \PGVt(u)+1+\|\bar{v}\|_{L^1(Q;\R^N)}<+\infty.
 \ee
 \end{proof}
 
 We prove that the infimum problem in the right-hand side of \eqref{eq:def-tgv-B} has a unique solution.
 \begin{proposition} \label{prop:min-exist}
 Let $u\in BV(Q)$. Then the infimum in \eqref{eq:def-tgv-B_k} is attained by a unique function $ v=(v_0,\ldots,v_{k-1})$, with $v_l\in BV_{\B_l}(Q;\mathbb M^{N^{l+1}})$ for $l=0,\ldots, k-1$.
 \end{proposition}
 \begin{proof}
 We start with the case $k=1$. Let $u\in BV(Q)$ and, without loss of generality, assume that $\alpha=(1,1)$. In view of Proposition \ref{weak_equi_norm} we have $\PGVt(u)<+\infty$.\\\\
The existence of a unique minimizer $v\in L^1(Q;\R^N)$ with $\B v\in \mb(Q;\mnn)$ follows from the Direct Method of the calculus of variations. Indeed, let $\seqn{v_n}\subset BV_{\B}(Q;\rn)$ be such that 
 \be
 \abs{Du-v_n}_{\mb(Q;\R^N)}+\abs{\B v_n}_{\mb(\Omega;\mnn)}\leq \PGVt(u)+1/n
 \ee
 for every $n\in \N$. Then,
\be\label{eq:bd-pn-l1}
 \norm{v_n}_{L^1(Q;\R^N)}\leq \abs{Du-v_n}_{\mb(Q;\R^N)}+\abs{Du}_{\mb(Q;\R^N)}\leq \PGVt(u)+\abs{Du}_{\mb(Q;\R^N)}+1/n,
\ee
and
\be\label{eq:bd-A-pn}
\abs{\B v_n}_{\mb(Q;\mnn)}\leq \PGVt(u)+1/n,
\ee
for every $n\in \N$. In view of Assumption \ref{assum_basic_B}, and together with \eqref{eq:bd-pn-l1} and \eqref{eq:bd-A-pn}, we obtain a function $v\in L^1(Q;\R^N)$ with $\B v\in \mb(Q;\mnn)$ such that, up to the extraction of a  subsequence (not relabeled), there holds
\be
v_n\to v\quad\text{strongly in }L^1(Q;\R^N),
\ee
and
\be
\liminfn\abs{\B v_n}_{\mathcal M_b(Q;\mnn)}\geq\abs{\B v}_{\mathcal M_b(Q;\mnn)}.
\ee
The minimality of $v$ follows by lower-semicontinuity, whereas the uniqueness is a consequence of the strict convexity of the $PGV_{\B}$-seminorm. \\\\
For $k=2$, we write $\B[2]=(\B_0,\B_1)$. Again without loss of generality, we assume that $(\alpha_0,\alpha_1,\alpha_2)=(1,1,1)$. Let $\seqn{v_0^n}\subset BV_{\B_0}(Q;\rn)$ and $\seqn{v_1^n}\subset BV_{\B_1}(Q;\mnn)$ be such that 
\begin{align}\label{v1_upper_bdd}
 \abs{Du-v_0^n}_{\mb(Q;\R^N)}+\abs{v^n_1-\B_0 v_0^n}_{\mb(Q;\mnn)}+\abs{\B_1 v_1^n}_{\mb(Q;\M^{N^3})}
 \leq PGV_{\B[2]}^{3}(u)+1/n.
\end{align}
We claim that 
\be
\label{eq:claim-num}
\sup\flp{\norm{v_1^n}_{L^1(Q;\mnn)}:\,\,n\in\N}<+\infty.
\ee
Suppose that the claim is false, i.e., that, up to the extraction of a subsequence (not relabeled), $\norm{v_1^n}_{L^1(Q;\mnn)}\to \infty$. By \eqref{v1_upper_bdd} we have that
\begin{align}
 &\frac{1}{\norm{v_1^n}_{L^1(Q;\mnn)}}\fmp{\abs{Du-v_0^n}_{\mb(Q;\R^N)}
 +\abs{v^n_1-\B_0 v_0^n}_{\mb(Q;\mnn)}+\abs{\B_1 v_1^n}_{\mb(Q;\M^{N^3})}}\\
 &\leq \frac1{\norm{v_1^n}_{L^1(Q;\mnn)}}\fsp{PGV_{\B[2]}^{3}(u)+1/n}\to 0.\label{eq:divided}
\end{align}
Defining
\be
\label{eq:def-auxiliary}
\tilde v_0^n:=v_0^n/\norm{v_1^n}_{L^1(Q;\mnn)}\text{ and }\tilde{v}_1^n:=v_1^n/\norm{v_1^n}_{L^1(Q;\mnn)},
\ee
we obtain that $\{\tilde v_0^n\}_{n=1}^{\infty}$ is bounded in $BV_{\B_0}(Q;\R^N)$ and $\{\tilde v_1^n\}_{n=1}^{\infty}$ is bounded in $BV_{\B_1}(Q;\R^{N\times N})$. In view of Assumption \ref{assum_basic_B}, there exist $\tilde v_0\in BV_{\B_0}(Q;\rn)$ and $\tilde v_1\in BV_{\B_1}(Q;\mnn)$ such that, upon the extraction of a subsequence (not relabeled), 
\begin{align}
&\tilde v_0^n\to{\tilde v_0}\quad\text{strongly in $L^1(Q;\rn)$ and $\text{weakly}^\ast$ in }BV_{\B_0}(Q;\rn),\\
&{\tilde v_1^n}\to{\tilde v_1}\quad\text{strongly in $L^1(Q;\mnn)$ and $\text{weakly}^\ast$ in }BV_{\B_1}(Q;\mnn).
\end{align}
In particular, since $\norm{\tilde v_i^n}_{L^1(Q;\mnn)}=1$ for every $n\in \mathbb{N}$, we deduce that $\norm{\tilde v_1}_{L^1(Q;\mnn)}=1$. By \eqref{eq:divided} we conclude that
\begin{align*}
&\norm{\tilde v_0}_{L^1(Q;\rn)}+\abs{\tilde v_1-\B_0\tilde v_0}_\mbqmnn+\abs{\B_1\tilde v_1}_{\mb(Q;\M^{N^3})}\\
&\,\,\,\,\,\,\leq \liminf_{n\to +\infty}\Big\{ \frac{1}{\norm{v_1^n}_{L^1(Q;\mnn)}}\big[\abs{Du-v_0^n}_{\mb(Q;\R^N)}\\
&\,\,\,\,\,\,\,\,\,\,\,\,\,+\abs{v^n_1-\B_0 v_0^n}_{\mb(Q;\mnn)}+\abs{\B_1 v_1^n}_{\mb(Q;\M^{N^3})}\big]\Big\}=0,
\end{align*}
which yields $\tilde v_0=0$ and $\tilde v_1=0$. This contradicts the fact that $\norm{\tilde v_1}_{L^1(Q;\mnn)}=1$ and implies \eqref{eq:claim-num}. \\

In view of \eqref{v1_upper_bdd} and \eqref{eq:claim-num} there holds
\be
\sup\flp{\norm{v_0^n}_{BV_{\B_0}(Q;\rn)}+\norm{v_1^n}_{BV_{\B_1}(Q;\mnn)}:\,\,n\in\N}<+\infty.
\ee
Thus, by Assumption \ref{assum_basic_B} there exist $v_0\in BV_{\B_0}(Q;\rn)$ and $v_1\in BV_{\B_1}(Q;\mnn)$ such that, again upon extracting a  subsequence (not relabeled), 
\begin{align}
& v_0^n\to{ v_0}\quad \text{strongly in $L^1(Q;\rn)$ and $\text{weakly}^\ast$ in }BV_{\B_0}(Q;\rn),\\
& v_1^n\to{ v_1}\quad \text{strongly in $L^1(Q;\mnn)$ and $\text{weakly}^\ast$ in }BV_{\B_1}(Q;\mnn).
\end{align}
In particular, $v_1^n-\B_0 v_0^n\wtos v_1-{\B_0 v_0}$ $\text{weakly}^\ast$ in $\mbqmnn$ and hence
\be
\liminfn\abs{v_1^n-\B_0 v_0^n}_{\mbqmnn}\geq\abs{v_1-{\B_0 v_0}}_{\mbqmnn}.
\ee
The minimality of $v:=(v_0,v_1)$ follows from \eqref{v1_upper_bdd} and by lower semicontinuity. The uniqueness is again a consequence of the strict convexity of the $PGV^{3}_{\B[2]}$-seminorm. We conclude the proof by remarking that the case for $k\geq3$ can be treated in an analogous way. 
\end{proof}
We close this section by studying the asymptotic behavior of the $PGV^{k+1}_{\B[k]}$ seminorms in terms of the operator $\B$ for subclasses of $\Pi_l$ satisfying Assumption \ref{assum_basic_BPi}.
\begin{proposition}\label{thm_asymptitic_PGV}
Let $k\in \mathbb{N}$ and let $u\in BV(Q)$. Let $\seqn{\B_n[k]}\subset \Pi^k$ and $\seqn{\alpha_n}\subset \R^{k+1}_+$ be such that $\B_n[k]\to \B[k]$ in $\ell^{\infty}$ and $\alpha_n\to \alpha\in\R^{k+1}_+$. Assume that $\seqn{\B_n[k]}$ satisfies Assumption \ref{assum_basic_BPi}. Then 
\be
\limn PGV_{\alpha_n,\B_n[k]}^{k+1}(u)=PGV_{\alpha,\B[k]}^{k+1}(u).
\ee
\end{proposition}
\begin{proof}
\noindent {Step 1: The case $k=1$}.
We present our argument in the case in which $\alpha_n:=(1,1)$ for all $n\in\N$. The general case can be argued in the same way since, by assumption, $\alpha\in \R^{2}_+$. 

We first claim that 
\be\label{liminf_u_const}
\liminfn \PGVon(u)\geq PGV^2_{\B}(u).
\ee
Indeed, by Proposition \ref{prop:min-exist} for each $n\in\N$ there exists $v_n\in BV_{\B_n}(Q;\rn)$ such that 
\be
\PGVon(u)=\abs{Du-v_n}_{\mb(Q;\R^N)}+\abs{\B_n v_n}_{\mb(Q;\mnn)}.
\ee
From \eqref{eq:ineq-bv} we see that 
\be
\abs{Du-v_n}_{\mb(Q;\R^N)}+\abs{\B_n v_n}_{\mb(Q;\mnn)}\leq \abs{Du}<+\infty,
\ee
which implies that $\seqn{v_n,\B_n}\subset L^1(Q;\rn)\times\Pi$ is bounded. Therefore, by Assumption \ref{assum_basic_BPi} there exist $\B\in\Pi$ and $v\in BV_{\B}(Q)$ such that $v_n\to  v$ strongly in $L^1(Q;\rn)$ and 
\be\label{eq_M_left_go}
\liminfn \abs{\B_n v_n}_{\mathcal M_b(Q;\mnn)}  \geq \abs{\B v}_{\mathcal M_b(Q;\mnn)}.
\ee
Thus, by \eqref{eq_M_left_go} we have
\begin{align}
\liminfn\, \PGVon(u)
&=\liminfn \fmp{\abs{Du-v_n}_{\mb(Q;\R^N)}+\abs{\B_n v_n}_{\mb(Q;\mnn)}}\\
&\geq \liminfn \abs{Du-v_n}_{\mb(Q;\R^N)}+\liminfn\abs{\B_n v_n}_{\mb(Q;\mnn)}\\
&\geq  \abs{Du-v}_{\mb(Q;\R^N)}+\abs{\B v}_{\mb(Q;\mnn)}\\
&\geq PGV_{\B}^2(u),
\end{align}
where in the last inequality we used \eqref{eq:def-tgv-B}. This concludes the proof of \eqref{liminf_u_const}.\\\\
We now claim that 
\be\label{limsup_u_const}
\limsupn \,\PGVon(u)\leq PGV^2_{\B}(u).
\ee
By Proposition \ref{prop:min-exist} there exists $v\in BV_{\B}(Q;\rn)$ such that 
\be
\PGVt(u)=\abs{Du-v}_{\mb(Q;\R^N)}+\abs{\B v}_{\mb(Q;\mnn)}.
\ee
In view of the density result in Assumption \ref{assum_basic_B}, Statement 2, we may assume that $v\in C^\infty(Q;\rn)$ and, for $\e>0$ small,
\be\label{smooth_0_approxi}
\PGVt(u)\geq\abs{Du-v}_{\mb(\Omega;\R^N)}+\abs{\B v}_{\mb(\Omega;\mnn)} -\e.
\ee
Since
\be
\PGVon(u)\leq \abs{Du-v}_{\mb(Q;\R^N)}+\abs{\B_n v}_{\mb(Q;\mnn)},
\ee
we obtain
\begin{align}
\limsupn\,\PGVon(u)&\leq \abs{Du-v}_{\mb(Q;\R^N)}+\limsupn \abs{\B_n v}_{\mb(Q;\mnn)}\\
&\leq \abs{Du-v}_{\mb(Q;\R^N)}+ \abs{\B v}_{\mb(Q;\mnn)}\\
&\leq \PGVt(u)+\e,
\end{align}
where in the last inequality we used \eqref{smooth_0_approxi}. Claim \eqref{limsup_u_const} is now asserted by the arbitrariness of $\e>0$.\\
\noindent{Step 2: The case $k\geq 2$}.
We will write the argument for $k=2$, the situation in which $k>2$ can be treated analogously. As in the setting $k=1$, we can assume that $\alpha_n:=(1,1,1)=:\alpha$ for every $n\in \mathbb{N}$. The proof of the inequality 
\be
\limsupn \,PGV_{\B_n[2]}^3(u)\leq PGV_{\B[2]}^3(u)
\ee
is similar to that in the case $k=1$. Therefore, we only need to show that
\be
\label{eq:lsc-k}
\liminfn PGV_{\B_n[2]}^3(u)\geq PGV_{\B[2]}^3(u).
\ee
If the left-hand side of \eqref{eq:lsc-k} is unbounded, then there is nothing to prove. Therefore, without loss of generality we can assume that there exists a constant $C>0$ such that
\be
\label{eq:bdd-k}
\liminfn\, PGV_{\B_n[2]}^3(u)\leq C.
\ee
Writing $\B_n[2]=(\B_n^0,\B_n^1)$, in view of Proposition \ref{prop:min-exist} there exist $v^n_0\in BV_{\B_n^0}(Q;\R^N)$ and $v^n_1\in BV_{\B_n^1}(Q;\mnn)$ such that
\be
\label{eq:pgvn-k}
PGV_{\B_n[2]}^3(u)=\abs{Du-v_n^0}_{\mb(Q;\R^N)}+\abs{\B_n^0 v_n^0-v_n^1}_{\mb(Q;\mnn)}+\abs{\B_n^1v_n^1}_{\mb(Q;\M^{N^3})}
\ee
for every $n\in \mathbb{N}$. We claim that
\be
\label{eq:unif-bdd-k}
\sup\flp{\norm{v_n^1}_{L^1(Q;\mnn)}:\,n\in \mathbb{N}}<+\infty.
\ee
Indeed, assume by contradiction that \eqref{eq:unif-bdd-k} is false. Then, upon extracting a  subsequence (not relabeled), there holds
\be
\limn\norm{v_n^1}_{L^1(Q;\mnn)}=+\infty.
\ee
Defining the two auxiliary sequences $\flp{\tilde v_n^0}$ and $\flp{\tilde v_n^1}$ as in \eqref{eq:def-auxiliary}, by \eqref{eq:bdd-k} and \eqref{eq:pgvn-k} we obtain
\be\label{eq:everything-zero}
\limn \abs{\frac{Du}{\norm{v_n^1}_{L^1(Q;\mnn)}}-\tilde v_n^0}_{\mb(Q;\R^N)}+\abs{\B_n^0 \tilde v_n^0-\tilde v_n^1}_{\mb(Q;\mnn)}+\abs{\B_n^1\tilde v_n^1}_{\mb(Q;\M^{N^3})}=0.
\ee
From this and the fact that $\flp{\B_n[2]}$ is bounded in $\ell^{\infty}$, by Assumption \ref{assum_basic_BPi} we conclude the existence of maps $\tilde v_0\in BV_{\B_0}(Q;\R^N)$ and $\tilde v_1\in BV_{\B_1}(Q;\mnn)$ such that, upon the extraction of a further  subsequence (not relabeled), there holds
\begin{align}
&\nonumber \tilde v_0^n\to{\tilde v_0} \text{ strongly in $L^1(Q;\rn)$ and $\B_n^0 \tilde v_0^n\to \B_0{\tilde v_0}\text{ weakly}^\ast$ in }\mb(Q;\rn),\\
&\label{eq:str-v1n}{\tilde v_1^n}\to{\tilde v_1}\text{ strongly in $L^1(Q;\mnn)$ and $\B_n^1 \tilde v_1^n\to \B_1{\tilde v_1}\text{ weakly}^\ast$ in }\mb(Q;\mnn).
\end{align}
By \eqref{eq:everything-zero} we infer that
\be
\limn \abs{\tilde v_0}_{\mb(Q;\R^N)}+\abs{\B_0 \tilde v_0-\tilde v_1}_{\mb(Q;\mnn)}+\abs{\B_1\tilde v_1}_{\mb(Q;\M^{N^3})}=0,
\ee
which, in turn, yields $\tilde v_0=0$ and $\tilde v_1=0$. This contradicts \eqref{eq:str-v1n}, and completes the proof of claim \eqref{eq:unif-bdd-k}.\\

By \eqref{eq:unif-bdd-k} and by Assumption \ref{assum_basic_BPi} we obtain the existence of maps $ v_0\in BV_{\B_0}(Q;\R^N)$ and $ v_1\in BV_{\B_1}(Q;\mnn)$ such that, upon the extraction of a further  subsequence (not relabeled), there holds
\begin{align}
&\nonumber v_0^n\to{ v_0}\quad \text{ strongly in $L^1(Q;\rn)$ and $\B_0  v_0^n\rightharpoonup^* \B_0{ v_0}\quad\text{ weakly}^\ast$ in }\mb(Q;\rn)\\
&{ v_1^n}\to{ v_1}\quad \text{ strongly in $L^1(Q;\mnn)$ and $\B_n^1  v_1^n\rightharpoonup^* \B_1{ v_1}\quad\text{ weakly}^\ast$ in }\mb(Q;\mnn).
\end{align}
Hence, by \eqref{eq:pgvn-k} and by lower-semicontinuity we conclude that
\begin{align*}
\liminfn\, PGV_{\B_n[2]}^3(u)&\geq \abs{Du-v_0}_{\mb(Q;\R^N)}+\abs{\B_0 v_0-v_1}_{\mb(Q;\mnn)}+\abs{\B_1v_1}_{\mb(Q;\M^{N^3})}\\
&\geq PGV_{\B[2]}^3(u),
\end{align*}
which concludes the proof of \eqref{eq:lsc-k} and of the proposition.
\end{proof}

\section{$\Gamma$-convergence of functionals defined by $PGV$- total generalized variation seminorms}\label{gamma_conv_sec}
In this section we prove a $\Gamma$-convergence result with respect to the operator $\B$. For $r>0$ we denote (see \eqref{BB_distance})
\be\label{small_around_B}
(\B)_r:=\flp{\B'\in \Pi:\,\,\norm{\B'-\B}_{\ell^\infty}\leq r}.
\ee
We recall the notation $\Pi^k=\Pi\times\Pi_2\times\cdots\times\Pi_k$ from Definition \ref{def_class_Pi}. Throughout this section let $u_{\eta}\in L^2(Q)$ be a given datum representing a corrupted image.

\begin{define}
\label{def:fractional TGV functional}
Let $k\in\N$, $\B[k]\in \Pi^{k}$, $\alpha\in \R^{k+1}_+$. We define the functional $\mathcal I^{k+1}_{\alpha,\B[k]}$ :$L^1(Q)\to [0,+\infty]$ as
\be
\mathcal I^{k+1}_{\alpha,\B[k]}(u):=
\begin{cases}
\norm{u-u_\eta}_{L^2(Q)}^2+ \PGVk(u)&\text{ if }u\in BV(Q),\\
+\infty &\text{ otherwise. }
\end{cases}
\ee
\end{define}

The following theorem is the main result of this section.
\begin{theorem}
\label{thm:new-Gamma}
Let $\seqn{\B_n[k]}\subset \Pi^{k}$ satisfy Assumption \ref{assum_basic_BPi}, and let $\seqn{\alpha_n}\subset \R^{k+1}_+$ be such that $\B_n[k]\to \B[k]$ in $\ell^{\infty}$ and $\alpha_n\to \alpha\in\R^{k+1}_+$. Then the functionals $\mathcal I^{k+1}_{\alpha_n,\B_n[k]}$ satisfy the following compactness properties:\\

{\rm (Compactness)} Let $u_n\in BV(Q)$, $n\in\N$, be such that 
\be
\sup\flp{\mathcal I^{k+1}_{\alpha_n,\B_n[k]}(u_n):\,\,n\in\N}<+\infty.
\ee
Then there exists $u\in BV(Q)$ such that, up to the extraction of a  subsequence (not relabeled),
\be
u_n\wtos  u\text{ weakly}^\ast\text{ in }BV(Q).
\ee
Additionally, $\mathcal I^{k+1}_{\alpha_n,\B_n[k]}$ $\Gamma$-converges to $\mathcal I^{k+1}_{\alpha,\B[k]}$ in the $L^1$ topology. To be precise, for every $u\in BV(Q)$ the following two conditions hold:\\

{\rm (Liminf inequality)} If 
\be u_n\to u\text{ in }L^1(Q)\ee
then 
\be
\mathcal I_{\alpha,\B[k]}^{k+1}(u)\leq \liminf_{n\to +\infty}\mathcal I_{\alpha_n,\B_n[k]}^{k+1}(u_n).\ee
{\rm (Recovery sequence)} For each $u\in BV(Q)$, there exists $\seqn{u_n}\subset BV(Q)$ such that 
\be u_n\to u\text{ in }L^1(Q)\ee
and 
\be
\limsup_{n\to +\infty}\mathcal I_{\alpha_n,\B_n[k]}^{k+1}(u_n)\leq \mathcal I_{\alpha,\B[k]}^{k+1}(u).
\ee
\end{theorem}
We subdivide the proof of Theorem \ref{thm:new-Gamma} into two propositions.\\\\
For $\B\in\Pi$, we consider the projection operator 
\be
\mathbb P_\B:L^1(Q;\rn)\to \mathcal N(\B).
\ee
Note that this projection operator is well defined owing to the assumption that $\mathcal N(\B)$ is finite dimensional (see \cite[page 38, Definition and Example 2]{brezis2010functional} and \cite[Subsection 3.1]{bbddlgftfb2017}).

Next we have an enhanced version of Korn's inequality.
\begin{proposition}\label{poincare_B_uniformly}
Let $\B\in\Pi$ and let $r>0$. Then there exists a constant $C=C(\B,Q)$, depending only on $\B$ and on the domain $Q$, such that 
\be\label{eq_poincare_B_uniformly}
\norm{v-\mathbb P_{\B'}(v)}_{L^1(Q;\rn)}\leq C \abs{\B' v}_{\mathcal M_b(Q;\mnn)},
\ee
for all $v\in L^1(Q)$ and $\B'\in (\B)_r$.
\end{proposition}

\begin{proof}
Suppose that \eqref{eq_poincare_B_uniformly} fails. Then there exist sequences $\seqn{\B_n}\subset (\B)_r$ and $\seqn{v_n}\subset L^1(Q)$ such that 
\be
\norm{v_n-\mathbb P_{\B_n}(v_n)}_{L^1(Q;\rn)}\geq n \abs{\B_n v_n}_{\mathcal M_b(Q;\mnn)}
\ee
for every $n\in \mathbb{N}$. Up to a normalization, we can assume that
\be\label{eq_poincare_B_uniformly2}
\norm{v_n-\mathbb P_{\B_n}(v_n)}_{L^1(Q;\rn)}=1\text{ and }\abs{\B_n v_n}_{\mathcal M_b(Q;\mnn)}\leq 1/n
\ee
for every $n\in \mathbb{N}$. Since $\seqn{\B_n}\subset (\B)_r$, up to a subsequence (not relabeled), we have $\B_n\to\tilde \B$ in $\ell^\infty$, for some $\tilde\B\in (\B)_r$. Next, let 
\be
\tilde v_n:= v_n-\mathbb P_{\B_n}(v_n).
\ee
Note that for each $n\in\N$
\be\label{eq_poincare_B_uniformly4}
\mathbb P_{\B_n}(\tilde v_n)=0.
\ee
Thus, by \eqref{eq_poincare_B_uniformly2} we have
\be\label{eq_poincare_B_uniformly3}
\norm{\tilde v_n}_{L^1(Q;\rn)}=1\text{ and }\abs{\B_n\tilde v_n}_{\mathcal M_b(Q;\mnn)}\leq1/n.
\ee
In view of Assumption \ref{assum_basic_BPi}, up to a further subsequence (not relabeled), there exists $\tilde v\in BV_{\tilde \B}(Q;\rn)$ such that $\tilde v_n\to \tilde v$ strongly in $L^1(Q)$ and $|\tilde\B\tilde v|_{\mathcal M_b(Q;\mnn)}=0$. Moreover, in view of \eqref{eq_poincare_B_uniformly3}, we also have $\norm{\tilde v}_{L^1(Q;\rn)}=1$.\\\\
Since the projection operator is Lipschitz with Lipschitz constant less than or equal to one, by \eqref{eq_poincare_B_uniformly4} we have
\be
\norm{\mathbb P_{\tilde \B}(\tilde v)}_{L^1(Q)}=\norm{\mathbb P_{\B_n}(\tilde v_n)-\mathbb P_{\tilde\B}(\tilde v)}_{L^1(Q)}\leq \norm{\tilde v-\tilde v_n}_{L^1(Q)} \to 0.
\ee
Thus, $\mathbb P_{\tilde\B}(\tilde v)=0$. However, $|\tilde\B\tilde v|_{\mathcal M_b(Q;\mnn)}=0$ implies that $\tilde v\in \mathcal N[\tilde\B]$ with $\tilde v=P_{\tilde\B}(\tilde v)$, and hence we must have $\tilde v=0$, contradicting the fact that $\norm{\tilde v}_{L^1(Q;\rn)}=1$.
\end{proof}
The following proposition is instrumental for establishing the liminf inequality.
\begin{proposition}\label{compact_semi_para}
Let $\seqn{\B_n[k]}\subset \Pi^k$ satisfy Assumption \ref{assum_basic_BPi}, and let $\seqn{\alpha_n}\subset \R^{k+1}_+$ be such that $\B_n[k]\to \B[k]$ in $\ell^{\infty}$ and $\alpha_n\to \alpha\in\R^{k+1}_+$. For every $n\in \N$ let $u_n\in BV(Q)$ be such that 
\be\label{liminf_unf_upper}
\sup\flp{\mathcal I^{k+1}_{\alpha_n,\B_n[k]}(u_n):\,\,n\in\N}<+\infty.
\ee
Then there exists $u\in BV(Q)$ such that, up to the extraction of a  subsequence (not relabeled),
\be\label{weak_BV_liminf}
u_n\wtos  u\text{ weakly}^\ast\text{ in }BV(Q)
\ee
and 
\be
\liminf_{n\to\infty}{PGV_{\alpha_n,\B_n[k]}^{k+1}(u_n)}\geq {PGV^{k+1}_{\alpha,\B[k]}(u)},
\ee
with
\be
\liminfn \,\mathcal I^{k+1}_{\alpha_n,\B_n[k]}(u_n)\geq \mathcal I^{k+1}_{\alpha,\B[k]}(u).
\ee
\end{proposition}
\begin{proof}
Without loss of generality  we assume that $k=1$ and that $\alpha_n:=(1,1)$ for every $n\in\N$, as the general case for $k>1$ and $\alpha\in \R_+^{k+1}$ can be argued with straightforward adaptations. \\\\
Fix $r>0$ and recall the definition of $(\B)_r$ from \eqref{small_around_B}. We claim that if $r$ is small enough then there exists $C_r>0$ such that 
\be\label{uinform_equivalent_norm}
\norm{u}_{BPGV^2_{\B'}(Q)}\leq\norm{u}_{BV(Q)}\leq C_r \norm{u}_{BPGV^2_{\B'}(Q)},
\ee
for all $u\in BV(Q)$ and $\B'\in (\B)_r$.\\\\
Indeed, by Definitions \ref{def-tgv-B} and \ref{def:BPGV} we always have
\be
\norm{u}_{BPGV^2_{\B'}(Q)}\leq\norm{u}_{BV(Q)},
\ee
for all $\B'\in\Pi$ and $u\in BV(Q)$. \\\\
The crucial step is to prove that the second inequality in \eqref{uinform_equivalent_norm} holds. Set
\be
\mathcal N_r(\B):=\{\om\in L^1(Q;\R^N):\,\text{ there exists }\,\B'\in (\B)_r\,\text{ for which }\,\om\in \mathcal N(\B')\}.
\ee
We claim that there exists $C>0$, depending on $r$, such that for each $u\in BV(Q)$ and $\om\in \mathcal N_r(\B)$ we have
\be\label{C_1_acquire}
\abs{Du}_{\mbqrn}\leq C\fsp{\abs{Du-\om}_{\mbqrn}+\norm{u}_{L^1(Q)}}.
\ee

Suppose that \eqref{C_1_acquire} fails. Then we find sequences $\seqn{u_n}\subset BV(Q)$ and $\seqn{\om_n}\subset \mathcal N_r(\B)$ such that 
\be
\abs{Du_n}_{\mbqrn}\geq n\fsp{\abs{Du_n-\om_n}_{\mbqrn}+\norm{u_n}_{L^1(Q)}}
\ee
for every $n\in \mathbb{N}$. Thus, up to a normalization, we can assume that
\be\label{shrink_unom3}
\abs{Du_n}_\mbqrn=1
\ee
and 
\be\label{shrink_unom}
\abs{Du_n-\om_n}_\mbqrn+\norm{u_n}_{L^1(Q)}\leq 1/n,
\ee
which implies that $u_n\to 0$ strongly in $L^1(Q)$ and 
\be\label{shrink_unom2}
\abs{Du_n-\om_n}_\mbqrn\to 0.
\ee
By \eqref{shrink_unom3} and \eqref{shrink_unom}, it follows that $\abs{\om_n}_\mbqrn$ is uniformly bounded, and hence, up to a subsequence (not relabeled), there exists $\om\in \mbqrn$ such that $\om_n\wtos \om$ in $\mbqrn$. 
For every $n\in\N$ let $\B_n'\in(\B)_r$ be such that $\om_n\in\mathcal N(\B_n')$. Then $\B_n' \om_n=0$ for all $n\in\N$. 
Since $\norm{\B_n'-\B}_{\ell^{\infty}}<r$, in particular the sequence $\seqn{\om_n,\B_n'}\subset L^1(\mbqrn)\times\Pi$ fulfills Assumption \ref{assum_basic_BPi}, and hence, upon extracting a further subsequence (not relabeled), there holds 
\be
\om_n\to\om_0\quad\text{strongly in }L^1(Q;\R^N).
\ee
 Additionally, since $u_n\to 0$ strongly in $L^1(Q)$, we infer that $Du_n\to 0$ in the sense of distributions. Therefore, by \eqref{shrink_unom2} we deduce that $\om_0=0$. Using again \eqref{shrink_unom}, we conclude that  
\be
\abs{Du_n}_{\mathcal M_b(Q;\rn)}\to 0,
\ee
which contradicts  \eqref{shrink_unom3}. This completes the proof of \eqref{C_1_acquire}.\\\\
We are now ready to prove  the second inequality in \eqref{uinform_equivalent_norm}, i.e., 
\be
\label{eq:sec-to-prove}
\norm{u}_{BV(Q)}\leq C_r \norm{u}_{BPGV^2_{\B'}(Q)}
\ee
for some constant $C_r>0$, and for all $\B'\in (\B)_r$.\\\\
Fix $\B'\in (\B)_r$, and by Proposition \ref{prop:min-exist} let $v_{\B'}$ satisfy
\be\label{C_tau_equal_B}
PGV_{\B'}^2(u)=\abs{Du-v_{\B'}}_{\mb(Q;\R^N)}+\abs{\B' v_{\B'}}_{\mb(Q;\mnn)}.
\ee

Since $\mathbb P_{\B'}[v_{\B'}]\in \mathcal N_r(\B)$, we have 
\begin{align*}
\abs{Du}_\mbqrn&\leq C(\abs{Du-\mathbb P_{\B'}[v_{\B'}]}_\mbqrn+\norm{u}_{L^1(Q)})\\
&\leq C(\abs{Du-v_{\B'}}_\mbqrn+\abs{v_{\B'}-\mathbb P_{\B'}[v_{\B'}]}_\mbqrn+\norm{u}_{L^1(Q)})\\
&\leq C(\abs{Du-v_{\B'}}_\mbqrn+C'\abs{{\B'} v_{\B'}}_{\mb(Q;\mnn)}+\norm{u}_{L^1(Q)})\\
&\leq (C+C')\fmp{\abs{Du-v_{\B'}}_\mbqrn+\abs{{\B'} v_{\B'}}_{\mb(Q;\mnn)}+\norm{u}_{L^1(Q)}}\\
&=(C+C') \fmp{PGV_{\B'}^2(u)+ \norm{u}_{L^1(Q)}},
\end{align*}
where in the first inequality we used \eqref{C_1_acquire}, the third inequality follows by \eqref{eq_poincare_B_uniformly}, and in the last equality we invoked \eqref{C_tau_equal_B}. Defining $C_r:=C+C'+1$, we obtain
\be
\norm{u}_{BV(Q)}=\norm{u}_{L^1(Q)}+\abs{Du}_{\mathcal M_b(Q;\rn)}\leq C_r (PGV_{\B'}^2(u)+ \norm{u}_{L^1(Q)})=C_r \norm{u}_{BPGV_{\B'}^2(Q)}
\ee
and we conclude \eqref{eq:sec-to-prove}.\\\\
Now we prove the compactness property. In view of \eqref{liminf_unf_upper} we have 
\be\label{uinform_equivalent_norm222}
\sup\flp{\norm{u_n}_{BPGV^2_{\B_n}(Q)}:\,\, n\in\N}<+\infty.
\ee
Since $\B_n\to\B$ in $\ell^{\infty}$, choosing $r=1$ there exists $N>0$ such that $\B_n\subset (\B)_1$ for all $n\geq N$. Thus, by  \eqref{uinform_equivalent_norm} and \eqref{uinform_equivalent_norm222}, we infer that
\be
\sup\flp{\norm{u_n}_{BV(Q)}:\,\, n\in\N}\leq C_1 \sup\flp{\norm{u_n}_{BPGV^2_{\B_n}(Q)}:\,\, n\in\N}<+\infty,
\ee
and thus we may find $ u\in BV(Q)$ such that, up to a subsequence (not relabeled), $u_n\wtos  u$ in $BV(Q)$.\\\\
Additionally, again from Proposition \ref{prop:min-exist}, for every $n\in\N$ there exists $v_n\in BV_{\B_n}(Q;\rn)$ such that,
\be
PGV_{\B_n}^2(u_n)=\abs{Du_n-v_{n}}_{\mb(Q;\R^N)}+\abs{\B_n v_n}_{\mb(Q;\mnn)}.
\ee
By \eqref{liminf_unf_upper} and  \eqref{weak_BV_liminf}, and in view of Assumption \ref{assum_basic_BPi}, we find $ v\in BV_{ \B}(Q;\rn)$ such that, up to a subsequence (not relabeled), $v_n\to  v$ strongly in $L^1$. Therefore, we have
\begin{align*}
\liminfn PGV_{\B_n}^2(u_n)&\geq\liminfn\abs{Du_n-v_{n}}_{\mb(Q;\R^N)}+\liminfn\abs{\B_n v_n}_{\mb(Q;\mnn)}\\
&\geq \abs{D  u- v}_{\mb(Q;\R^N)}+\abs{\B v}_{\mb(Q;\mnn)}\\
&\geq PGV^2_{\B}( u),
\end{align*}
where in the second to last inequality we used Assumption \ref{assum_basic_BPi}. This concludes the proof of the proposition.
\end{proof}
\begin{proposition}\label{new_equal}
Let $\seqn{\B_n[k]}\subset \Pi^k$ satisfy Assumption \ref{assum_basic_BPi}, and let $\seqn{\alpha_n}\subset \R^{k+1}_+$ be such that $\B_n[k]\to \B[k]$ in $\ell^{\infty}$ and $\alpha_n\to \alpha\in\R^{k+1}_+$. Then for every $u\in BV(Q)$ there exists $\seqn{u_n}\subset BV(Q)$ such that $u_n\to u$ in $L^1$ and
\be
\limsup_{n\to\infty}{PGV_{\alpha_n,\B_n[k]}^{k+1}(u_n)}\leq {PGV_{\alpha,\B[k]}^{k+1}(u)}.
\ee
\end{proposition}
\begin{proof}
This is a direct consequence of Proposition \ref{thm_asymptitic_PGV} by choosing $u_n:=u$.
\end{proof}

We close Section \ref{gamma_conv_sec} by proving Theorem \ref{thm:new-Gamma}.
\begin{proof}[Proof of Theorem \ref{thm:new-Gamma}]
Properties {\rm(Compactness)} and (Liminf inequality) hold in view of Proposition \ref{compact_semi_para}, and Property (Recovery sequence) follows from Proposition \ref{new_equal}. 
\end{proof}

\section{The bilevel training scheme with $PGV$- regularizers}\label{sec_ts_PGV}
Let $u_\eta\in L^2(Q)$ and $u_c\in BV(Q)$ be the corrupted and clean images, respectively. In what follows we will refer to pairs $(u_c,u_\eta)$ as {training pairs}. We recall that $\Pi$ was introduced in Definition \ref{def_class_Pi}. 
\begin{define}\label{def_training_set}
For every $k\in \N$, we say that $\Sigma \subset \Pi^k$ is a {training set} if the operators in $\Sigma$ satisfy Assumption \ref{assum_basic_BPi}, and if $\Sigma$ is closed and bounded in $\ell^{\infty}$. 
\end{define}

Examples of training sets for $k=1$ (where $\Pi^1=\Pi$) are provided in Section \ref{sec_explicit_example}. We introduce the following bilevel training scheme. 
\begin{define}
Let $\theta\in(0,1)$ and $k\in\N$, and let $\Sigma$ be a training set. The two levels of the scheme $(\mathcal T^{k+1}_\theta)$ are 
\begin{flalign}
{\text{Level 1. }}&\,\,\,\,\,\,(\tilde \alpha,\tilde\B[k])\in \argmin\flp{\norm{u_c-u_{\alpha,\B}}_{L^2(Q)}^2:\,\,\alpha\in[\theta,1/\theta]^{k+1},\,\,\B[k]=(\B_0,\ldots,\B_{k-1})\in \Sigma},&\\
{\text{Level 2. }}&\,\,\,\,\,\,u_{\alpha,\B[k]}:=\argmin\flp{\norm{u-u_\eta}_{L^2}^2+ \PGVk(u),\,\, u\in BV(Q)}.&
\end{flalign}

In what follows we will only focus on the case $k=1$, as the cases in which $k>1$ follow via straightforward modifications. For convenience of notation, we remark that our training scheme for $k=1$, $(\mathcal T^2_\theta)$, can be described as follows:
\begin{flalign}
{\text{Level 1. }}&\,\,\,\,\,\,\,\,\,\,\,\,\,\,\,\,\,\,\,\,\,\,\,\,\,\,\,\,\,\,\,\,\,\,(\tilde \alpha,\tilde\B):= \argmin\flp{\norm{u_c-u_{\alpha,\B}}_{L^2(Q)}^2:\,\,\alpha\in[\theta,1/\theta]^2,\,\,\B\in\Sigma},\tag{$\mathcal T_\theta^2$-L1}\label{ABtraining_0_1} &\\
{\text{Level 2. }}&\,\,\,\,\,\,\,\,\,\,\,\,\,\,\,\,\,\,\,\,\,\,\,\,\,\,\,\,\,\,\,\,\,\,u_{\alpha,\B}:=\argmin\flp{\norm{u-u_\eta}_{L^2(Q)}^2+ \PGV(u),\,\, u\in BV(Q)}.\tag{$\mathcal T_\theta^2$-L2}\label{ABsolution_map}&
\end{flalign}
\end{define}

We first show that the Level 2 problem in \eqref{ABsolution_map} admits a unique solution for every given $u_\eta\in L^2(Q)$. 
\begin{proposition}\label{unique_exist_lower}
Let $u_\eta\in L^2(Q)$. Let $\B\in\Sigma$, and let $\alpha\in\R^2_+$. Then there exists a unique $u_{\alpha,\B}\in BV(Q)$ such that
\be
\|\uab-u_\eta\|^2_{L^2(Q)}+\PGVt(\uab)=\min\flp{\norm{u-u_\eta}^2_{L^2(Q)}+\PGVt(u):\, u\in BV(Q)}.
\ee
\end{proposition}
\begin{proof}
As before, we assume that $\alpha:=(1,1)$. Let $\seqn{u_n}\subset BV(Q)$ be such that
\be\label{eq:bd-n}\norm{u_n-u_{\eta}}^2_{L^2(Q)}+\PGVt(u_n)\leq \inf \flp{\|u-u_{\eta}\|^2_{L^2(Q)}+\PGVt(u):\, u\in BV(Q)}+1/n,
\ee
for every $n\in \N$, and let $\{v_n\}\subset BV_{\B}(Q)$ be the associated sequence of maps provided by Proposition \ref{prop:min-exist}. In view of \eqref{eq:bd-n}, there exists a constant $C$ such that
\be\label{eq:bd-n-pn}
\|u_n-u_\eta\|^2_{L^2(Q)}+\abs{Du_n-v_n}_{\mb(Q;\rn)}+ \abs{\B v_n}_{\mb(Q;\mnn)}\leq C
\ee
for every $n\in \N$. We claim that
\be\label{eq:unif-bd-pn}
\sup\flp{\|v_n\|_{L^1(Q;\R^N)}:\,\, n\in\N}<+\infty.
\ee
Indeed, if \eqref{eq:unif-bd-pn} does not hold, then, up to the extraction of a subsequence (not relabeled), we have
$$\lim_{n\to +\infty}\norm{v_n}_{L^1(Q;\rn)}= +\infty.$$
Setting 
\be\label{eq:def-rescaled}
\tilde{u}_n:=\frac{u_n}{\|v_n\|_{L^1(Q;\rn)}}\quad\text{and }\quad\tilde{v}_n:=\frac{v_n}{\|v_n\|_{L^1(Q;\rn)}}\quad\text{for every }n\in \N,
\ee
and dividing both sides of \eqref{eq:bd-n-pn} by $\|v_n\|_{L^1(Q)}$, we deduce that
\begin{align}\label{eq:main-rescaled}
&\lim_{n\to +\infty}\left[{\norm{\tilde{u}_n-\frac{u_{\eta}}{\|v_n\|_{L^1(Q;\rn)}}}^2_{L^2(Q)}+\abs{D\tilde{u}_n-\tilde{v}_n}_{\mb(Q;\rn)}+ \abs{\B \tilde{v}_n}_{\mb(Q;\mnn)}}\right]=0.
\end{align}
In view of \eqref{eq:def-rescaled} and \eqref{eq:main-rescaled}, and by Assumption \ref{assum_basic_BPi}, there exists $\tilde{v}\in BV_{\B}(Q;\rn)$, with 
\be\label{eq:norm-p-tilde}\|\tilde{v}\|_{L^1(Q;\R^N)}=1,\ee
such that
\be
\label{eq:strong-vn-v}
\tilde{v}_n\to \tilde{v}\quad\text{strongly in }L^1(Q;\R^N),
\ee
and
\be\label{eq:strong-vn-v2}
\B \tilde{v}_n\wtos \B \tilde{v}\quad\text{weakly}^\ast\text{ in }\mb(Q;\mnn).
\ee
Additionally, \eqref{eq:main-rescaled} and \eqref{eq:strong-vn-v} yield
\be\label{eq:u-tilde-n}
\tilde{u}_n\to 0\quad\text{strongly in }L^2(Q),
\ee
and
\be\label{eq_upp_down}
\limsup_{n\to +\infty} \abs{D\tilde{u}_n-\tilde{v}}_{\mb(Q;\rn)}\leq \lim_{n\to +\infty}\abs{D\tilde{u}_n-\tilde{v}_n}_{\mb(Q;\rn)}+\lim_{n\to +\infty}\norm{\tilde{v}_n-\tilde{v}}_{L^1(Q;\rn)}=0.
\ee
Since by \eqref{eq:u-tilde-n} $D\tilde u_n\to 0$ in the sense of distribution, we deduce from \eqref{eq_upp_down} that $\tilde v =0$. This contradicts \eqref{eq:norm-p-tilde}, and implies claim \eqref{eq:unif-bd-pn}. \\\\
By combining \eqref{eq:bd-n-pn} and \eqref{eq:unif-bd-pn}, we obtain the uniform bound 
\be
\abs{Du_n}_{\mb(Q;\rn)}\leq \abs{Du_n-v_n}_{\mb(Q;\rn)}+\|v_n\|_{L^1(Q;\R^N)}\leq C
\ee
for every $n\in \N$ and some $C>0$. Thus, by \eqref{eq:bd-n-pn} and Assumption \ref{assum_basic_B} there exist $u_\B\in BV(Q)$ and $v\in BV_{\B}(Q)$ such that, up to the extraction of a subsequence (not relabeled), 
\begin{align*}
&u_n\wto u_\B\quad\text{weakly in }L^2(Q),\\
&u_n\wtos u_\B\quad\text{weakly}^\ast\text{ in }BV(Q),\\
&v_n\to v\quad\text{strongly in }L^1(Q;\R^N),\\
&\B v_n\wtos \B v\text{ weakly}^\ast\text{ in }\mb(Q;\mnn).
\end{align*}

In view of \eqref{eq:bd-n}, and by lower-semicontinuity, we obtain the inequality
\begin{align*}
\norm{u_\B-u_0}^2_{L^2(Q)}+\abs{Du_\B-v}_{\mb(Q;\rn)}+ \abs{\B v}_{\mb(Q;\mnn)}\\
\leq \inf\flp{\norm{u-u_\eta}^2_{L^2(Q)}+PGV_\B^2(u):\, u\in BV(Q)}.
\end{align*}
The uniqueness is a consequence of the strict convexity of the $\PGV$ - seminorm.
\end{proof}
\begin{theorem}\label{main_thm}
Let $\theta\in(0,1)$ be fixed. Then the training scheme $(\mathcal T_\theta^{k+1})$ admits at least one solution $(\ta, \tilde \B[k])\in[\theta,1/\theta]^{k+1}\times \Sigma^k$, and provides an associated optimally reconstructed image $u_{\ta,\tilde \B[k]}\in BV(Q)$.
\end{theorem}
\begin{proof}
Again we only treat the case in which $k=1$. The case $k>1$ can be dealt with similarly. Let $\seqn{\alpha_n,\B_n}\subset [\theta,1/\theta]^2\times \Sigma$ be a minimizing sequence obtained from \eqref{ABtraining_0_1}. By the boundedness and closedness of $\Sigma$ in $\ell^{\infty}$, up to a subsequence (not relabeled), there exists $(\ta, \tilde \B)\in[\theta,1/\theta]^2\times \Sigma$ such that $\alpha_n\to\ta$ in $\R^2$ and $\B_n\to \tilde\B$ in $\ell^{\infty}$. Therefore, in view of Theorem \ref{thm:new-Gamma} and the properties of $\Gamma$-convergence, we have 
\be\label{main_thm_eq}
u_{\alpha_n,\B_n}\wtos u_{\ta,\tilde \B}\text{ weakly}^\ast\text{ in }BV(Q)\text{ and strongly in }L^1(Q),
\ee
where $u_{\alpha_n,\B_n}$ and $u_{\ta,\tilde \B}$ are defined in \eqref{ABsolution_map}. \\\\
By \eqref{main_thm_eq}, we have 
\be
\norm{u_{\ta,\tilde \B}-u_c}_{L^2(Q)}\leq \liminfn \norm{u_{\alpha_n,\B_n}-u_c}_{L^2(Q)},
\ee
which completes the proof.
\end{proof}

\section{Training set $\Sigma[\A]$ based on $(\A,\B)$ training operators pairs}\label{fsaqs_sec}
This section is devoted to providing a class of operators $\B$ belonging to $\Pi$ (see Definition \ref{def_class_Pi}), satisfying Assumption \ref{assum_basic_BPi}, and being closed with respect to the convergence in \eqref{BB_distance}.
\subsection{A subcollection of $\Pi$ characterized by $(\A,\B)$ training operators pairs}
\
\phantom{AAAAA} Let $U$ be an open set in $\R^{N}$, and let $\A:\mathcal D'(U;\rn)\to \mathcal D'(U;\rn)$ be a $d$-th order differential operator, defined as
\be
\A u:= \sum_{\abs{a}\leq d}A_{a} \frac{\partial^a}{\partial x^a}u\quad\text{for every }\,u\in \mathcal D'(U;\rn),
\ee
where, for every multi-index $a=(a^1,a^{2},\ldots,a^N)\in \N^N$,
\be
\frac{\partial^a}{\partial x^a}:=\frac{\partial^{a^1}}{\partial x_1^{a^1}}\frac{\partial^{a^{2}}}{\partial x_2^{a^{2}}}\cdots \frac{\partial^{a^N}}{\partial x_N^{a^N}}
\ee
is meant in the sense of distributional derivatives, and $A_{a}$ is a linear operator mapping from $\rn$ to $\rn$. 
Let $\B$ be a first order differential operator, $\B:\mathcal D'(U;\R^N)\to \mathcal D'(U;\mathbb R^{N\times N})$, given by 
\be\label{A_quasiconvexity_operator-again}
\mathscr B v:=\sum_{i=1}^N B^i\frac{\partial}{\partial x_i} v\quad\text{for every }v\in \mathcal D'(U;\R^N),
\ee
where $B^i\in \mathbb M^{N^{3}}$ for each $i=1,\dots,N$, and where $\frac{\partial}{\partial x_i}$ denotes the distributional derivative with respect to the $i$-th variable.
We will restrict our analysis to elliptic pairs $(\A,\B)$ satisfying the ellipticity assumptions below.
\begin{define}\label{ready_AB_to_work}
We say that $(\A,\B)$ is a {training operator pair} if $\B$ has finite dimensional null-space $\mathcal{N}(\B)$, and $(\A,\B)$ satisfies the following assumptions:
\begin{enumerate}[1.]
\item
For every $\lambda\in \flp{-1,1}^N$, the operator $\A$ has a fundamental solution $P_{\lambda}\in L^1(\rn;\rn)$ such that: 
\begin{enumerate}[$a$.]
\item 
$\A P_{\lambda} = \lambda\delta$, where $\delta$ denotes the Dirac measure centered at the origin;
\item
$P_{\lambda}\in C^{\infty}(\rn\setminus\{0\};\rn)$ and $\frac{\partial^{a}}{\partial x^{a}} P_{\lambda}\in L^1(\rn;\rn)$ for every multi-index $a\in \N^N$ with $|a|\leq d-1$ (where $d$ is the order of the operator $\pdeor$); 
\item
for every $a\in \N^N$ with $|a|\leq d-1$, and for every open set $U\subset \R^N$ such that $Q\subset U$, we have
\be\label{cuoweixiangjian}
\sum_{\abs{a}= d-1} \norm{\tau_h\fsp{\frac{\partial^{a}}{\partial x^{a}} \Pl}-\frac{\partial^{a}}{\partial x^{a}} P_{\lambda}}_{L^1(U;\R^N)}=:M_\A(U;h)\to0\quad\text{as}\, \,\abs{h}\to 0,
\ee
where for $h\in\rn$, the translation operator $\tau_h:L^1(\rn;\rn)\to L^1(\rn;\rn)$ is defined by
\be\label{eq:def-tau-h}
\tau_h w(x):=w(x+h)\quad\text{for every }w\in L^1(\rn;\rn)\,\text{and for a.e.}\, x\in \rn.
\ee
\end{enumerate}
\item  For every open set $U\subset \R^N$ such that $Q\subset U$, and for every $u\in W^{d,1}(U;\rn)$ and $v\in C^{\infty}_c(U;\rn)$  
\begin{align}
\label{IBP_calcu_assum}
& \norm{(\pdeor u)_i\ast v_i}_{L^1(U)}\leq C_{\A}\fmp{\sum_{\abs{a}\leq d-1}\norm{\frac{\partial^{a}}{\partial x^{a}}u}_{L^1(U;\rn)}}\abs{\B v}_{\mathcal M_b(U;\mnn)},
\end{align}
for every $i=1,\ldots, N$, where the constant $C_{\A}$ depends only on the operator $\A$. The same property holds for $u\in C^{\infty}_c(U;\rn)$ and $v\in BV_{\B}(U;\rn)$ (see \eqref{BVB_norm}).
\end{enumerate}
\end{define}

Explicit examples of operators $\A$ and $\B$ satisfying Definition \ref{ready_AB_to_work} are provided in Section \ref{sec_explicit_example}. Condition $2.$ in Definition \ref{ready_AB_to_work} can be interpreted as an ``integration by parts-requirement", as highlighted by the example below. Let $N=2$, $d=2$, $\B=\nabla$, and let $U\subset \R^2$ be an open set such that $Q\subset U$. Consider the following second order differential operator
\be
\pdeor u:=\fsp{\frac{\partial^2 u_1}{\partial x_1^2}\quad \frac{\partial^2 u_2}{\partial x_2^2}}^{\intercal}\quad\text{for every}\quad u=(u_1,u_2)^{\intercal}\in D'(U;\R^2).
\ee
Then, for every $u\in W^{2,1}(U;\R^2)$ and $v\in C^{\infty}_c(U;\R^2)$ there holds
\begin{align*}
 \norm{(\pdeor u)_i\ast v_i}_{L^1(U)}&=  \norm{\frac{\partial^2 u_i}{\partial x_i^2}\ast v_i}_{L^1(U)}=\norm{\frac{\partial u_i}{\partial x_i}\ast \frac{\partial v_i}{\partial x_i}}_{L^1(U)}\leq \norm{\nabla u}_{L^1(U;\R^{2\times 2})}\norm{\nabla v}_{L^1(U;\R^{2\times 2})}\\
& =\norm{\nabla u}_{L^1(U;\R^{2\times 2})}\norm{\B v}_{L^1(U;\R^{2\times 2})},
\end{align*}
for every $i=1,2$. In other words, the pair $(\pdeor,\B)$ satisfies \eqref{IBP_calcu_assum} with $C_{\pdeor}=1$.
\begin{define}\label{def_PiA}
For every $\A$ as in Definition \ref{ready_AB_to_work} we denote by $\Pi_\A$ the following collection of first order differential operators $\B$,
\be
\Pi_\A:=\flp{\B:\,\, \text{$(\A,\B)$ is a training operator pair}}.
\ee
\end{define}
The first result of this section is the following. 
\begin{theorem}\label{thm_PiA_eq_Pi}
Let $\A$ be as in Definition \ref{ready_AB_to_work}. Let $\Pi$ and $\Pi_\A$ be the collections of first order operators introduced in Definition \ref{def_class_Pi} and Definition \ref{def_PiA}, respectively. Then
\be
\Pi_\A\subset \Pi,
\ee
thus every operator $\B\in\Pi_\A$ satisfies Assumption \ref{assum_basic_B}. Additionally, the operators in $\Pi_{\A}$ fulfill Assumption \ref{assum_basic_BPi}.
\end{theorem}

We proceed by first recalling two preliminary results from the literature. The next proposition, that may be found in {\cite[Theorem 4.26]{brezis2010functional}}, will be instrumental in the proof of a regularity result for distributions with bounded $\B$-total-variation (see Proposition \ref{fdmt_thm_BV}).
\begin{proposition}\label{lp_compactness}
Let $\mathcal F$ be a bounded set in $L^p(\rn)$ with $1\leq p<+\infty$. Assume that 
\be
\lim_{\abs{h}\to 0}\norm{\tau_h f-f}_{L^p(\rn)}=0\text{ uniformly in }\mathcal F.
\ee
Then, denoting by $\mathcal F\lfloor_{Q}$ the collection of the restrictions to $Q$ of the functions in $\mathcal F$, the closure of $\mathcal F\lfloor_{Q}$ in $L^p(Q)$ is compact.
\end{proposition}

We also recall some basic properties of the space $BV_\B(Q;\rn)$ for $\B\in\Pi_\A$ (see \cite[Section 2]{bbddlgftfb2017}) .
\begin{proposition}\label{smooth_approx_h}
Let  $\B\in \Pi_{\A}$. Let $U$ be an open set in $\R^N$. Then 
\begin{enumerate}[1.]
\item
$BV_\B(U;\rn)$ is a Banach space with respect to the norm defined in \eqref{BVB_norm2};
\item
$C^\infty(U,\rn)$ is dense in $BV_\B(U;\rn)$ in the strict topology, i.e., for every $u\in BV_\B(U;\rn)$ there exists $\seqn{u_n}\subset C^\infty(U,\rn)$ such that 
\be
u_n\to u\text{ strongly in }L^1(U;\rn)\text{ and }\abs{\B u_n}_{\mathcal M_b(U;\mnn)}\to \abs{\B u}_{\mathcal M_b(U;\mnn)}.
\ee
\end{enumerate}
\end{proposition}
Before we establish Theorem \ref{thm_PiA_eq_Pi}, we prove a technical lemma.
\begin{lemma}\label{le_auigao_ind}
Let $k\in\N$. Then there exists a constant $C>0$ such that, for every $h\in \rn$ and $w\in W_\loc^{k,1}(\rn;\rn)$, there holds 
\be
\limsup_{\abs{h}\to 0}\sum_{\abs{a}\leq k}\norm{\tau_h\Big(\frac{\partial^{a}}{\partial x^{a}}w\Big)-\frac{\partial^{a}}{\partial x^{a}}w}_{L^1(Q;\rn)}\leq \limsup_{\abs{h}\to 0}C \sum_{\abs{a}= k}\norm{\tau_h\Big(\frac{\partial^{a}}{\partial x^{a}}w\Big)-\frac{\partial^{a}}{\partial x^{a}}w}_{L^1(Q;\rn)},
\ee 
where $\tau_h$ is the operator defined in \eqref{eq:def-tau-h}. 
\end{lemma}
\begin{proof}
By the linearity of $\tau_h$, we have
\be
\tau_h\Big(\frac{\partial^{a}}{\partial x^{a}}w\Big)-\frac{\partial^{a}}{\partial x^{a}}w = \frac{\partial^{a}}{\partial x^{a}}(\tau_hw-w).
\ee
On the one hand, by the Sobolev embedding theorem (see, e.g., \cite{leoni}), we have
\begin{align}\label{auigao_ind}
&\sum_{\abs{a}\leq k}\norm{\tau_h\Big(\frac{\partial^{a}}{\partial x^{a}}w\Big)-\frac{\partial^{a}}{\partial x^{a}}w}_{L^1(Q;\rn)}\\
&=\sum_{\abs{a}\leq k}\norm{ \frac{\partial^{a}}{\partial x^{a}}(\tau_hw-w)}_{L^1(Q;\rn)}\\
&\leq C\norm{\tau_h(w)-w}_{L^1(Q;\rn)}+C\sum_{\abs{a}= k}\norm{ \frac{\partial^{a}}{\partial x^{a}}(\tau_hw-w)}_{L^1(Q;\rn)}.
\end{align}

On the other hand, by the continuity of the translation operator in $L^1$ (see, e.g., \cite[Lemma 4.3]{brezis2010functional} for a proof in $\R^N$, the analogous argument holds on bounded open sets) we have 
\be\label{auigao_ind2}
\limsup_{\abs{h}\to 0} \norm{\tau_h(w)-w}_{L^1(Q;\rn)}=0.
\ee
The result follows by combining \eqref{auigao_ind} and \eqref{auigao_ind2}.
\end{proof}

The next proposition shows that operators in $\Pi_\A$ satisfy Assumption \ref{assum_basic_B}.
\begin{proposition}\label{fdmt_thm_BV}
Let $\B\in\Pi_\A$, and let $BV_\B(Q;\rn)$ be the space introduced in Definition \ref{def_BVB}. 
Then the injection of $BV_\B(Q;\rn)$ into $L^1(Q;\rn)$ is compact.
\end{proposition}
\begin{proof}
 In view of Proposition \ref{smooth_approx_h}, for every $u\in BV_\B(Q;\R^N)$ there exists a sequence of maps $\{u^n\}_{n=1}^{\infty}\subset C^{\infty}(Q;\R^N)$ such that
\be
\label{eq:dense}
\|u^n-u\|_{L^1(Q;\R^N)}+\abs{\norm{\B u^n}_{L^1(Q;\mnn)}-\abs{\B u}_{\mathcal M_b(Q;\mnn)}}\leq \frac{1}{n}.
\ee
With a slight abuse of notation, we still denote by $u^n$ the $C^d$-extension of the above maps to the whole $\R^N$ (see e.g. \cite{fefferman}), where $d$ is the order of the operator $\mathscr A$. Without loss of generality, up to a multiplication by a cut-off function, we can assume that $u^n\in C^d_c(2Q;\R^N)$ for every $n\in \N$.
\\

We first show that, setting
\be
\mathcal F:=\flp{u\in L^1(Q;\rn):\,\, \norm{u}_{BV_\B(Q;\rn)}\leq 1},
\ee
 for every $n\in \mathbb{N}$ there holds
\be
\label{eq:claim-n}
\lim_{\abs{h}\to 0}\sup_{u\in \mathcal F}\flp{\norm{\tau_h {u}^n-{u}^n}_{L^1(Q;\rn)}}=0,
\ee
where we recall $\tau_h$ from Theorem \ref{lp_compactness}, and where for fixed $u\in \mathcal F$, $u^n$ is as above and satisfying \eqref{eq:dense}.\\\\

Let $h\in\rn$ and let $\delta_h$ be the Dirac distribution centered at $h\in\rn$. By the properties of the fundamental solution $P_{\lambda}$ we deduce
\begin{align}
\tau_h(\lambda_i u_i^n) &= \delta_h\ast \lambda_i u_i^n=\delta_h\ast(\lambda_i\delta\ast u_i^n)=\delta_h\ast\fsp{(\A \Pl)_i\ast u_i^n}\\
&=\fsp{\delta_h\ast (\A \Pl)_i}\ast u_i^n=\fsp{\A\fsp{\delta_h\ast (\Pl)}}_i\ast u_i^n,
\end{align}
for every $i=1,\ldots, N$, and every $\lambda\in \flp{-1,1}^N$. Therefore, we obtain that 
\begin{align}\label{cuowei_minus}
&\norm{\tau_h( \lambda_i u_i^n)-\lambda_i u_i^n}_{L^1(Q;\rn)}\\
&=\norm{\fsp{\A\fsp{\delta_h\ast (\Pl)}}_i\ast u_i^n-\fsp{\A \Pl}_i\ast u_i^n}_{L^1(Q;\rn)}=\norm{\fsp{\A\fsp{\delta_h\ast (\Pl)-\Pl}}_i\ast u_i^n}_{L^1(Q;\rn)}\\
&\leq C_\A\fmp{\sum_{\abs{a}\leq d-1} \norm{\tau_h\fsp{\frac{\partial^{a}}{\partial x^{a}}\Pl}-\frac{\partial^{a}}{\partial x^{a}}\Pl}_{L^1(Q;\rn)}}\abs{\B u^n}_{\mathcal M_b(Q;\mnn)}
\end{align}
for every $\lambda\in \{-1,1\}^N$, where in the last inequality we used the fact that $\tau_h\Pl-\Pl\in W^{d-1,d}(\R^N;\R^N)$ owing to Definition \ref{ready_AB_to_work}, Assertion 1c, the identity $\tau_h\fsp{\frac{\partial^{a}}{\partial x^{a}}\Pl}=\frac{\partial^{a}}{\partial x^{a}}\fsp{\tau_h\Pl}$, as well as
Definition \ref{ready_AB_to_work}, Assertion 2.\\\\
In particular, choosing $\bar{\lambda}:=(1,\dots,1)$ we have
\begin{align}
&\sup_{u\in\mathcal F}\flp{\norm{\tau_h( {u}^n)-{u}^n}_{L^1(Q;\rn)}}\\
&\quad\leq C_\A\Big(1+\frac1n\Big) \sum_{\abs{a}\leq d-1} \norm{\tau_h\fsp{\frac{\partial^{a}}{\partial x^{a}}P_{\bar{\lambda}}}-\frac{\partial^{a}}{\partial x^{a}}P_{\bar{\lambda}}}_{L^1(Q;\rn)},
\end{align}
and, in view of \eqref{cuoweixiangjian} and Lemma \ref{le_auigao_ind}, we conclude that
\begin{align}
&\lim_{\abs{h}\to 0}\sup_{u\in\mathcal F}\flp{\norm{\tau_h(u^n)-u^n}_{L^1(Q;\rn)}} \\
&\leq C_\A\Big(1+\frac1n\Big) \lim_{\abs{h}\to 0}\sum_{\abs{a}= d-1} \norm{\tau_h\Big(\frac{\partial^{a}}{\partial x^{a}}P_{\bar{\lambda}}\Big)-\frac{\partial^{a}}{\partial x^{a}}P_{\bar{\lambda}}}_{L^1(Q;\rn)}=0
\end{align}
for every $n\in \mathbb{N}$, which yields \eqref{eq:claim-n}.\\

By \eqref{eq:dense}, for $n\in \mathbb{N}$ fixed, for every $h\in \R^N$ with $|h|<1$, and for every $u\in \mathcal{F}$ there holds
\begin{align*}
&\norm{\tau_h u-u}_{L^1(Q;\rn)}\leq \norm{\tau_h u-\tau_h u^n}_{L^1(Q;\rn)}+\norm{\tau_h u^n-u^n}_{L^1(Q;\rn)}+\norm{u^n-u}_{L^1(Q;\rn)}\\
&\quad\leq \frac1n+\norm{u}_{L^1(Q_{|h|};\R^N)}+\norm{u^n}_{L^1(Q_{|h|};\R^N)}+ \norm{\tau_h u^n-u^n}_{L^1(Q;\rn)}\leq \frac2n+\sup_{v\in \mathcal F}\flp{\norm{\tau_h(v^n)-v^n}_{L^1(Q;\rn)}},
\end{align*}
where we have still denoted by $u$ the extension of the above map to zero on $\R^N\setminus Q$, and where $Q_{|h|}:=\fsp{-\frac12-|h|,\frac12+|h|}^N\setminus \fsp{-\frac12+|h|,\frac12-|h|}^N$.
By \eqref{eq:claim-n}, and since $L^1(Q_{|h|})\to 0$ as $|h|\to 0$, we deduce
\be
\lim_{|h|\to 0}\sup_{u\in\mathcal F} \flp{\norm{\tau_h(u)-u}_{L^1(Q;\rn)}}\leq \frac2n,
\ee
and letting $n\to +\infty$ we get
\be
\lim_{|h|\to 0}\sup_{u\in\mathcal F} \flp{\norm{\tau_h(u)-u}_{L^1(Q;\rn)}}=0.
\ee
Thus, recalling that $u=0$ on $\R^N\setminus Q$, we deduce the estimate
\begin{align*}
&\lim_{|h|\to 0}\sup_{u\in\mathcal F} \flp{\norm{\tau_h(u)-u}_{L^1(\R^N;\rn)}}\\
&\quad\leq \lim_{|h|\to 0}\sup_{u\in\mathcal F} \flp{\norm{\tau_h(u)-u}_{L^1(Q;\rn)}+C\norm{u}_{L^1(Q_{|h|};\R^N)}}=0.
\end{align*}
The statement now follows from Proposition \ref{lp_compactness}.
\end{proof}
The following extension result in $BV_{\B}$ is a corollary of the properties of the trace operator defined in \cite[Section 4]{bbddlgftfb2017}.
\begin{lemma}
\label{lemma:extension}
Let $\B\in\Pi_\A$, and let $BV_\B(Q;\rn)$ be the space introduced in Definition \ref{def_BVB}. Then there exists a continuous extension operator $\mathbb{T}:\,BV_\B(Q;\rn)\to BV_\B(\rn;\rn)$ such that $\mathbb{T}u=u$ almost everywhere in $Q$ for every $u\in BV_\B(Q;\rn)$.
\end{lemma}
\begin{proof}
Since $\mathcal{N}(\B)$ is finite dimensional, in view of \cite[(4.9) and Theorem 1.1]{bbddlgftfb2017} there exists a continuous trace operator ${\rm tr}\,:\,BV_\B(Q;\rn)\to L^1(\partial Q;\rn)$. By the classical results by E. Gagliardo (see \cite{gagliardo}) there exists a linear and continuous extension operator ${\rm E}: L^1(\partial Q;\rn)\to W^{1,1}(\R^N\setminus Q;\rn)$. The statement follows by setting
$$\mathbb{T}u:=u\chi_{Q}+{\rm E}({\rm tr}(u))\chi_{\R^N\setminus Q},$$
where $\chi_{Q}$ and $\chi_{\R^N\setminus Q}$ denote the characteristic functions of the sets $Q$ and $\R^N\setminus Q$, respectively, and by Theorem \cite[Corollary 4.21]{bbddlgftfb2017}.
\end{proof}
We point out that, as a direct consequence of Lemma \ref{lemma:extension}, we obtain 
\be\label{full_extension_eq}
\abs{\B (\mathbb Tu)}_{\mb(\rn;\mnn)}\leq C \abs{\B u}_{\mbqmnn},
\ee
where the constant $C$ depends only on $Q$ and $\abs{\B}_{\ell^\infty}$.

We close this subsection by proving a compactness and lower-semicontinuity result for functions with uniformly bounded $BV_{\B_n}$ norms. We recall that the definition of $M_\A$ is found in \eqref{cuoweixiangjian} .
\begin{proposition}\label{cpt_lsc_ub}
Let $\seqn{\B_n}\subset \Pi_{\A}$ be such that $\B_n\to\B$ in $\ell^{\infty}$. For every $n\in\N$ let $v_n\in BV_{\B_n}(Q;\R^N)$ be such that 
\be\label{B_n_uniform_bdd}
\sup\flp{\norm{v_n}_{BV_{\B_n}(Q;\R^N)}:\,\,n\in\N}<+\infty.
\ee
Then there exists $v\in BV_{\B}(Q;\R^N)$ such that, up to a subsequence (not relabeled), 
\be
\label{eq:extra-sss}
v_n\to v\text{ strongly in }L^1(Q;\R^N),
\ee
and 
\be\label{eq_B_n_uniform_bdd22}
{\B_nv_n}\wtos {\B v}\text{ $\text{weakly}^\ast$ in }\mb(Q;\mnn).
\ee
\end{proposition}
\begin{proof}
Let ${v_n}$ satisfy \eqref{B_n_uniform_bdd}. With a slight abuse of notation we still indicate by $v_n$ the $BV_{\B}$ continuous extension of the above maps to $\R^N$ (see Lemma \ref{lemma:extension}). Let $\phi\in C^{\infty}_c(2Q;\R^N)$ be a cut-off function such that $\phi\equiv 1$ on $Q$, and for every $n\in \N$ let $\tilde{v}_n$ be the map $\tilde{v}_n:=\phi v_n$. Note that ${\rm supp}\, \tilde{v}_n\subset\subset 2Q$. Additionally, by Lemma \ref{lemma:extension} there holds
\begin{align}
\label{eq:prop-ext}
\norm{\tilde{v}_n}_{BV_{\B}(2Q;\rn)}&{\leq} \norm{v_n}_{L^1(2Q;\rn)}{+}\abs{\B v_n}_{M_b(2Q;\R^{N\times N})}\\
&\quad\notag{+}\norm{\sum_{i=1}^N B^i \frac{\partial\phi}{\partial x_i}}_{L^{\infty}(2Q;\mathbb M^{N^3})}\norm{v^n}_{L^1(2Q;\R^N)}\\
&\notag\leq C_1\norm{{v}_n}_{BV_{\B}(2Q;\rn)}\leq C_2 \norm{{v}_n}_{BV_{\B}(Q;\rn)},
\end{align}
where in the last inequality we used Lemma \ref{lemma:extension}, and where the constants $C_1$ and $C_2$ depend only on the cut-off function $\phi$. To prove \eqref{eq:extra-sss} we first show that
\be
\label{eq:claim-n-new}
\lim_{\abs{h}\to 0}\sup_{n\in \N}\flp{\norm{\tau_h \tilde{v}_n-\tilde{v}_n}_{L^1(\rn;\rn)}}=0,
\ee
where we recall $\tau_h$ from Theorem \ref{lp_compactness}. Arguing as in the proof of \eqref{cuowei_minus}, by \eqref{eq:prop-ext} we deduce that for $|h|$ small enough, since $\rm{supp}\,\phi\subset\subset 2Q$,
\begin{align*}
&\norm{\tau_h \tilde{v}_n-\tilde{v}_n}_{L^1(\rn;\rn)}=\norm{\tau_h \tilde{v}_n-\tilde{v}_n}_{L^1(2Q;\rn)}\\
&\quad\leq C\fmp{\sum_{\abs{a}\leq d-1} \norm{\tau_h\fsp{\frac{\partial^{a}}{\partial x^{a}}\Pl}-\frac{\partial^{a}}{\partial x^{a}}\Pl}_{L^1(2Q;\rn)}}\abs{\B \tilde{v}_n}_{\mathcal M_b(2Q;\mnn)}\\
&\quad\leq C\fmp{\sum_{\abs{a}\leq d-1} \norm{\tau_h\fsp{\frac{\partial^{a}}{\partial x^{a}}\Pl}-\frac{\partial^{a}}{\partial x^{a}}\Pl}_{L^1(2Q;\rn)}}\norm{v_n}_{BV_{\B}(Q;\mnn)}
\end{align*}
for every $n\in \N$. Property \eqref{eq:claim-n-new} follows by \eqref{cuoweixiangjian}. Owing to Proposition \ref{lp_compactness}, we deduce \eqref{eq:extra-sss}.\\\\

We now prove \eqref{eq_B_n_uniform_bdd22}. Let $\vp\in C_c^\infty(Q;\mnn)$ be such that $\abs{\vp}\leq 1$. Then
\begin{align}
\limn\int_Q\vp\cdot d(\B_nv_n) &= \limn \sum_{i,j=1}^N \int_Q \vp_{ij}d\fsp{\sum_{k,l=1}^N (B_n)_{ijl}^k\frac{\partial (v_n)_l}{\partial x_k}}\\
&= \limn \sum_{i,j,k,l=1}^N \int_Q \vp_{ij}d\fsp{ (B_n)_{ijl}^k\frac{\partial (v_n)_l}{\partial x_k}}\\
&=- \limn \sum_{i,j,k,l=1}^N \int_Q (v_n)_{l}{ (B_n)_{ijl}^k\frac{\partial \vp_{ij}}{\partial x_k}}dx\\
&=-  \sum_{i,j,k,l=1}^N \int_Q v_{l}{ (B)_{ijl}^k\frac{\partial \vp_{ij}}{\partial x_k}}dx
\end{align}
where in the last step we used the fact that $v_n\to v$ strongly in $L^1(Q)$ and $\B_n\to\B$ in $\ell^{\infty}$.\\\\
This completes the proof of \eqref{eq_B_n_uniform_bdd22} and of the proposition.
\end{proof}

\begin{proof}[Proof of Theorem \ref{thm_PiA_eq_Pi}]
Let $\B\in\Pi_\A$ be given. The fact that $\B$ satisfies Assumption \ref{assum_basic_B} follows by Propositions \ref{smooth_approx_h} and \ref{fdmt_thm_BV}. The fulfillment of Assumption \ref{assum_basic_BPi} is a direct consequence of Proposition \ref{cpt_lsc_ub}.
\end{proof}

\subsection{Training scheme with fixed and multiple operators $\A$}
We first introduce a collection $\Sigma[\A]$ for a given operator $\A$ of order $d\in\N$.
\begin{define}\label{training_set_def}
Let $\A$ be a differential operator of order $d\in\N$. For every $\e>0$ we denote by $\Sigma_{\e}[\A]$ the collection 
\be
\Sigma_{\e}[\A]:=\flp{\B\in \Pi_\A:\e\leq \norm{\B}_{\ell^{\infty}}\leq 1}.
\ee
\end{define}
The first result of this subsection is the following.
\begin{theorem}\label{thm_close_training_set}
Fix $\e>0$. Let $\A$ be a differential operator of order $d\in\N$ such that $\Sigma_{\e}[\A]$ is non-empty. Then the collection $\Sigma_{\e}[\A]$ is a training set (see Definition \ref{def_training_set}).
\end{theorem}
\begin{proof}
By the definition of $\Sigma_{\e}[\A]$ we just need to show that $\Sigma_{\e}[\A]$ is closed in $\ell^{\infty}$. Let $u\in C^\infty(Q;\R^N)$ and $\seqn{\B_n}\subset \Sigma_{\e}[\A]$ be given. Then, up to a subsequence (not relabeled), we may assume that $\B_n\to \B$ in $\ell^{\infty}$. We claim that $\B\in \Pi_\A$. \\\\
To prove that $\mathcal N(\B)$ is finite-dimensional, we recall that this condition is equivalent to the injectivity of $\mathbb B(\xi)$ for all $\xi\in \mathbb C^N\setminus \flp{0}$ (see \cite[Remark 2.1]{bbddlgftfb2017}). Since for all $\xi\in \mathbb C^N\setminus\flp{0}$ we have that $\mathbb B(\xi)$ is the uniform limit of the sequence of injective linear maps $\seqn{\mathbb B_n(\xi)}$, either $\mathbb B(\xi)$ is constant or it is injective. On the other hand, the linearity of $\mathbb B(\xi)$ implies that it is constant only if it is identically zero. The fact that $\e\leq \norm{\B}_{\ell^{\infty}}\leq 1$ for all $n\in\N$ guarantees that this cannot occur, and yields the injectivity of $\mathbb B(\xi)$ and hence the fact that the dimension of $\mathcal N(\B)$ is finite. \\\\
To conclude the proof of the theorem we still need to show that $(\A,\B)$ satisfies Definition \ref{ready_AB_to_work}, Assertion 2. Let $U$ be an open set in $\R^N$ such that $Q\subset U$. Let $u\in C^{\infty}_c(U;\mnn)$ and let $v\in BV_{\B}(U;\rn)$. By Proposition \ref{smooth_approx_h} there exists $\seqk{v_k}\subset C^\infty(U;\rn)$ such that 
\be\label{strong_l1_strick}
v_k\to v\text{ strongly in }L^1(U;\rn)\text{ and }\abs{\B v_k}_{\mathcal M_b(U;\R^{N\times N})}\to \abs{\B v}_{\mathcal M_b(U;\R^{N\times N})}.
\ee
Integrating by parts we obtain
\begin{align}
&\nonumber \norm{\fsp{\pdeor u}_i\ast (v_k)_i}_{L^1(U;\rn)}\leq C_{\A}\fmp{\sum_{\abs{a}\leq d-1}\norm{\frac{\partial^{a}}{\partial x^{a}}u}_{L^1(U;\rn)}}\abs{\B_n v_k}_{\mathcal M_b(U;\mnn)},
\end{align}
for every $i=1,\ldots, N$. Taking the limit as $n\to\infty$ first, and then as $k\to\infty$, since $\B_n\to\B$ in $\ell^{\infty}$ and in view of \eqref{strong_l1_strick}, we conclude that 
\begin{align}
&\nonumber \norm{\fsp{\pdeor u}_i\ast (v_k)_i}_{L^1(U;\rn)}\leq C_{\A}\fmp{\sum_{\abs{a}\leq d-1}\norm{\frac{\partial^{a}}{\partial x^{a}}u}_{L^1(U;\rn)}}\abs{\B v}_{\mathcal M_b(U;\mnn)}.
\end{align}
The proof of the second part of Assertion 2 is analogous.
This shows that $(\A,\B)$ satisfies Definition \ref{ready_AB_to_work} and concludes the proof of the theorem.
\end{proof}
\begin{remark}
We note that the result of Theorem \ref{thm_close_training_set} still holds if we replace the upper bound $1$ in Definition \ref{training_set_def} with an arbitrary positive constant.
\end{remark}
We now consider the case of multiple operators $\A$.
\begin{define}
We say that collection $\mathcal A$ of differential operators $\A$ is a {training set builder} if 
\be\label{all_in_one_in}
\sup\flp{C_\A:\,\, \A\in\mathcal A}<+\infty\text{ and }\lim_{\abs{h}\to0} \sup\flp{M_\A(h):\,\, \A\in\mathcal A}=0,
\ee
where $C_\A$ and $M_\A(h)$ are defined in \eqref{IBP_calcu_assum} and \eqref{fund_mental_const}, respectively. \\\\
For every $\e>0$ we then define the class $\Sigma_{\e}[\mathcal A]$ via
\be
\Sigma_{\e}[\mathcal A]:=\operatorname{convex}\,\,\operatorname{hull}\fsp{\bigcup_{\A\in\mathcal A} \Sigma_{\e}[\A]},
\ee
where for every $\A\in\mathcal A$, $\Sigma_{\e}[\A]$ is the class defined in Definition \ref{training_set_def}.
\end{define}
We close this section by proving the following theorem.
\begin{theorem}\label{main_thm_K}
Let $\mathcal A$ be a training set builder. Then $\Sigma_{\e}[\mathcal A]$ is a training set.
\end{theorem}
\begin{proof}
The proof of this theorem follows the argument in the proof of Theorem \ref{thm_close_training_set} using the fact that the two critical constants  $M_\A(h)$ and $C_\A$, in \eqref{cuoweixiangjian} and \eqref{IBP_calcu_assum}, respectively, are uniformly bounded due to \eqref{all_in_one_in}.
\end{proof}

\section{Explicit examples and numerical observations}\label{sec_explicit_example}
In this section we exhibit several explicit examples of operators $\A$ and training sets $\Sigma_{\e}[\A]$, we provide numerical simulations and some observations derived from them.

\subsection{The existence of fundamental solutions of operators $\A$}

One important requirement in Definition \ref{ready_AB_to_work} is the existence of the fundamental solution $\Pl\in L^1(\rn,\rn)$ of a given operator $\A$. 
A result in this direction can be found in {\cite[Page 351, Section 6.3]{hsiao2008boundary}}, where an explicit form of the fundamental solution for Agmon-Douglis-Nirenberg elliptic systems with constant coefficients is provided.

\begin{remark}\label{rmk_ADNES}
In the case in which $N=2$, $\A$ has order $2$ and satisfies the assumptions in \cite[Page 351, Section 6.3]{hsiao2008boundary}, the fundamental solution $\Pl$ can be written as 
\be\label{eq_rmk_ADNES}
\Pl(x,y)=\frac{1}{8\pi^2}(\Delta L_y) \int_{\abs{\eta}=1,\eta\in\R^2}\fsp{(x-y)\cdot \eta}^2 \log \abs{ (x-y)\cdot \eta} R_\A d\om_\eta,
\ee
where $L$ denotes the fundamental solution of Laplace's equation, $R_\A$ denotes a constant depending on $\A$, and the integration is taken over the unit circle $\abs{\eta} = 1$ with arc length element $d\omega_\eta$.\\\\
In the special case in which 
\be
\label{eq:special-A}
\A w:=\Delta w+\nabla(\divg w)\quad\text{  for $w\in\mathcal D'(Q;\rnt)$,}
\ee
the fundamental solution $P_{\alpha}$, with $\A P_{\alpha}=\alpha\delta$ for $\alpha\in\R^2$, is given by
\be
P_{\alpha}(x):=\frac{3\alpha}{8\pi}\log \frac1{\abs{x}}+\frac x{8\pi}\frac{\alpha\cdot \abs{x}}{\abs{x}^2}.
\ee
We observe that $\nabla P_{\alpha}$ is positively homogeneous of degree $-1(= 1-N)$. Also, since $R_\A$ in \eqref{eq_rmk_ADNES} is a constant, $\nabla \Pl$ must have the same homogeneity as $\nabla P_{\alpha}$, which is $1-N$.
\end{remark}

\begin{proposition}\label{ready_AB_to_work_add}
Let $\A$ be a differential operator of order $d\in\N$, and assume that its fundamental solution $\Pl$ is such that $\frac{\partial^{a}}{\partial x^{a}}\Pl$ is positively homogeneous of degree $1-N$ for all multi-indexes $a\in\N^N$ with $\abs{a}=d-1$. Then Assertion 1c. in Definition \ref{ready_AB_to_work} is satisfied.
\end{proposition}
\begin{proof}
Let $s\in(0,1)$ be fixed. Since $\frac{\partial^{a}}{\partial x^{a}}P_\lambda$ is positively homogeneous of degree $1-N$ for all multi-indexes $a\in \N^N$ with $\abs{a}=d-1$, by \cite[Lemma 1.4]{temam1985mathematical} we deduce the estimate 
\begin{align}\label{est_N_1_up}
&\sum_{\abs{a}=d-1}\abs{\tau_h\fsp{\frac{\partial^{a}}{\partial x^{a}} \Pl(x)}-\frac{\partial^{a}}{\partial x^{a}} \Pl(x)}\\
&\quad\leq C\fmp{\max\flp{\sup\flp{\abs{\nabla^{d-1} \Pl(z)}:\,\, {\abs{z}=1}},\,\, \sup\flp{\abs{\nabla^d\Pl(z)}:\,\, {\abs{z}=1}}}}\cdot\\
&\qquad\cdot \abs{h}^{s}\fmp{\frac1{\abs{x}^{N-1+s}}+\frac1{\abs{x+h}^{N-1+s}}}.
\end{align}
for every $x\in \R^N$, $0\leq s\leq 1$, and $\abs{h}\leq 1/2$, where the constant $C$ is independent of $x$ and $h$.\\\\
Next, for every bounded open set $U\subset \R^N$ satisfying $Q\subset U$ we have
\begin{align}\label{cuowei_minus2}
&\int_U\frac1{\abs{x}^{N-1+s}}dx\leq \int_{B(0,2)}\frac1{\abs{x}^{N-1+s}}dx+\int_{U\setminus B(0,2)}\frac1{\abs{x}^{N-1+s}}dx\\
&\leq 2\pi\int_0^2 r^{-s}dr+\frac{1}{2^{N-1+s}}|U\setminus B(0,2)|<+\infty,
\end{align}
The analogous computation holds for $\frac{1}{\abs{x+h}^{N-1+s}}$.
Since $P_\lambda$ is a fundamental solution and $\A P_\lambda=\lambda\delta$, we have that $P_\lambda\in C^\infty(\rn\setminus B(0,\e))$ for every $\e>0$. In particular, 
\be\label{fund_mental_const}
\max\flp{\sup\flp{\abs{\nabla^{d-1} P(z)}:\,\, {\abs{z}=1}},\,\, \sup\flp{\abs{\nabla^dP(z)}:\,\, {\abs{z}=1}}}=:M <+\infty.
\ee
This, together with \eqref{est_N_1_up} and \eqref{cuowei_minus2}, yields
\be
\norm{\sum_{\abs{a}=d-1}\abs{\tau_h\fsp{\frac{\partial^{a}}{\partial x^{a}} P_\lambda(x)}-\frac{\partial^{a}}{\partial x^{a}} P_\lambda(x)}}_{L^1(U;\R^N)}\leq C M\abs{h}^s ,
\ee
for some $C>0$, and thus
\be
\lim_{h\to\infty}\norm{\sum_{\abs{a}=d-1}\abs{\tau_h\fsp{\frac{\partial^{a}}{\partial x^{a}} P_\lambda(x)}-\frac{\partial^{a}}{\partial x^{a}} P_\lambda(x)}}_{L^1(U;\R^N)}=0,
\ee
and \eqref{cuoweixiangjian} is established.
\end{proof}

\begin{remark}\label{all_A_work_remark}
As a corollary of Proposition \ref{ready_AB_to_work_add} and Remark \ref{rmk_ADNES}, we deduce that all operators $\A$ satisfying the assumptions in \cite[Page 351, Section 6.3]{hsiao2008boundary} comply with Definition \ref{ready_AB_to_work}, Assertion 1.
\end{remark}

\subsection{The unified approach to $TGV^2$ and $NsTGV^2$ - an example of $\Sigma[\A]$} \label{uatta_sec}
In this section we give an explicit construction of an operator $\A$ such that the seminorms $NsTGV^2$ and $TGV^2$, as well as a continuum of topologically equivalent seminorms connecting them, can be constructed as operators $\B\in\Sigma[\A]$.\\\\
We start by recalling the definition of the classical symmetrized gradient,
\be\label{symm_grad_example}
\mathcal E v = \frac{\nabla v+(\nabla v)^T}{2}=
\begin{bmatrix}
\partial_1 v_1 &\frac{(\partial_1 v_2+\partial_2 v_1)}{2}\\
\frac{(\partial_1 v_2+\partial_2 v_1)}{2} & \partial_2 v_2
\end{bmatrix},
\ee
for $v=(v_1,v_2)\in C^\infty(Q;\R^2)$. Let
\be
B^1_{\operatorname{sym}}=\fmp{
\begin{array}{@{}cc|cc@{}}
1 & 0 & 1/2 & 0\\
0& 1/2 & 0& 0\\
\end{array}
}
\text{ and }
B^2_{\operatorname{sym}}=\fmp{
\begin{array}{@{}cc|cc@{}}
0 & 0 & 1/2 & 0\\
0& 1/2 & 0& 1\\
\end{array}
},
\ee
and let $\B_{\operatorname{sym}}(v)$ be defined as in \eqref{A_quasiconvexity_operator} with $B^1_{\operatorname{sym}}$ and $B^2_{\operatorname{sym}}$ as above. Then $\B_{\operatorname{sym}} (v) =\mathcal E v$ for all $v\in C^\infty(Q;\rnt)$, and $\mathcal N(\B_{\operatorname{sym}})$ is finite dimensional. In particular, 
\be
\mathcal N(\B_{\operatorname{sym}})=\flp{v(x)=\alpha\Big(\begin{array}{c}x_2\\-x_1\end{array}\Big)+b:\,\alpha\in \mathbb{R}\quad\text{and}\quad b\in \mathbb{R}^2}.
\ee 

The first part of Definition \ref{ready_AB_to_work} follows from Remark \ref{all_A_work_remark}. Next we verify that \eqref{IBP_calcu_assum} holds. Indeed, choosing $\A$ as in \eqref{eq:special-A}, we first observe that 
\begin{align}\label{IBP_computation_explicit}
(\Aep w)\ast v&= \sum_{j=1}^N\fmp{\Delta w_j+\partial_j \divg(w)}\ast v_j
=\sum_{i,j=1}^N (\partial_i w_j+\partial_j w_i)\ast \partial_i v_j \\
&= \sum_{i,j=1}^N(\partial_i w_j+\partial_j w_i)\ast (\partial_i v_j+\partial_j v_i) ={(\B_{\operatorname{sym}} w)\ast (\B_{\operatorname{sym}}v)},
\end{align}
for every $w\in W^{1,2}(Q;\rnt)$ and $v\in C^\infty_c(Q;\rnt)$. That is, for every open set $U\subset \R^N$ such that $Q\subset U$ we have 
\be
\abs{(\Aep w)\ast v}_{\mb(U;\R^2)}\leq \abs{(\B_{\operatorname{sym}} w)\ast (\B_{\operatorname{sym}}v)}_{\mb(U;\mathbb{M}^{2\times 2})}\leq \norm{\nabla w}_{L^1(U;\mnnt)}\abs{\B_{\operatorname{sym}}(v)}_{\mb(U;\mathbb{M}^{2\times 2})}.
\ee
The same computation holds for $w\in C^{\infty}_c(Q;\rnt)$ and $v\in BV_{\B}(Q;\rnt)$. This proves that Assertion 2 in Definition \ref{ready_AB_to_work} is also satisfied.\\\\
We finally construct an example of a training set $\Sigma[\Aep]$. For every $0\leq s,t\leq 1$, we define
\be
B_t:=\fmp{
\begin{array}{@{}cc|cc@{}}
1 & 0 & t & 0\\
0& (1-t) & 0& 0\\
\end{array}
}
\text{ and }
B_s:=\fmp{
\begin{array}{@{}cc|cc@{}}
1 & 0 & s & 0\\
0& 1-s & 0& 0\\
\end{array}
},
\ee
and we set
\be\label{skewed_symmetric_gradient}
\B_{s,t}( v): = B_t \partial_1 v+B_s\partial_2 v=
\begin{bmatrix}
\partial_1 v_1 &(1-t)\partial_1 v_2+(1-s)\partial_2 v_1\\
t\partial_1 v_2+s\partial_2 v_1 & \partial_2 v_2
\end{bmatrix}.
\ee
By a straightforward computation, we obtain that $\mathcal N(\B_{s,t})$ is finite dimensional for every $0\leq s,t\leq 1$. Additionally, Assertion 1 in Definition \ref{ready_AB_to_work} follows by adapting the arguments in Remark \ref{all_A_work_remark}. Finally, arguing exactly as in \eqref{IBP_computation_explicit}, we obtain that 
\be
(\Aep w)\ast v = (\B_{t,s} w)\ast (\B_{s,t} (v)),\text{ for every }w,v\in C^\infty(\bar Q;\rnt),
\ee
which implies that 
\begin{align}
\abs{(\Aep w)\ast v }_{\mb(Q;\R^2)}&\leq \norm{\B_{t,s} w}_{L^1(Q;\mathbb M^{2\times 2})}\abs{\B_{s,t} (v)}_\mbqmnnt\\
&\leq  2\norm{\nabla w}_{L^1(Q;\mnn)}\abs{\B_{s,t} (v)}_\mbqmnnt.
\end{align}
Hence, we deduce again Statement 2 in Definition \ref{ready_AB_to_work}. Therefore, the collection $\Sigma[\Aep]$ given by
\be
\Sigma[\Aep]:=\flp{\B_{s,t}:\,\, 0\leq s,t\leq 1}
\ee
is a training set according to Definition \ref{training_set_def}. We remark that $\Sigma[\Aep]$ includes the operator $TGV^2$ (with $s=t=1/2$) and the operator $NsTGV^2$ (with $t=0$ and $s=1$), as well as a collection of all ``interpolating" regularizers. In other words, our training scheme $(\mathcal T^2_{\theta})$ with training set $\Sigma[\Aep]$ is able to search for optimal results in a class of operators including the commonly used $TV$, $TGV^2$ and $NsTGV^2$, as well as any interpolation regularizer.
\subsubsection{Comparison with other works}\label{sce_better_other}
In \cite{2018arXiv180201895B} the authors analyze a range of first order linear operators generated by diagonal matrixes. To be precise, letting $D=\operatorname{diag}(\beta_1,\beta_2,\beta_3,\beta_4)$, \cite{2018arXiv180201895B} treats first order operators $\B$ defined as
\be
\B v := Q \cdot B \cdot Q\cdot (\nabla v)^T,
\ee
where
\be
Q:=
\begin{bmatrix}
0 & 1& -1& 0\\
1 &0 &0 &1\\
-1& 0& 0& 1\\
0 &1& 1& 0
\end{bmatrix}
\text{ and }
\nabla v = [\partial_1 v_1, \partial_1 v_2, \partial_2v_1,\partial_2v_2]. 
\ee
That is, instead of viewing $\nabla v$ as a $2\times 2$ matrix as we do, in \cite{2018arXiv180201895B} $\nabla v$ is represented as a vector in $\R^4$. In this way, the symmetric gradient $\mathcal Ev$ in \eqref{symm_grad_example} can be written as
\begin{align}
\mathcal Ev &= Q\cdot\operatorname{diag}(0,1/2,1/2,1/2)\cdot Q \cdot (\nabla v)^T
= 
\begin{bmatrix}
1 & 0& 0& 0\\
0 &1/2 &1/2 &0\\
0& 1/2& 1/2& 0\\
0 &0 & 0& 1
\end{bmatrix}
\cdot  [\partial_1 v_1, \partial_1 v_2, \partial_2v_1,\partial_2v_2]^T\\
& = [\partial_1 v_1, 0.5(\partial_1 v_2+\partial_2 v_1),0.5(\partial_1 v_2+\partial_2 v_1),\partial_2v_2].
\end{align}
However, the representation above does not allow to consider skewed symmetric gradients $\B_{s,t}(v)$ with the structure introduced in \eqref{skewed_symmetric_gradient}. Indeed, let $s=t=0.2$. We have 
\be
\B_{0.2,0.2}(v)=
\begin{bmatrix}
\partial_1 v_1 &0.8\partial_1 v_2+0.8\partial_2 v_1\\
0.2\partial_1 v_2+0.2\partial_2 v_1 & \partial_2 v_2
\end{bmatrix}.
\ee
Rewriting the matrix above as a vector in $\R^4$, we obtain
\begin{align}
\B_{0.2,0.2}(v)&=[\partial_1 v_1, 0.2(\partial_1 v_2+\partial_2 v_1),0.8(\partial_1 v_2+\partial_2 v_1),\partial_2v_2]\\
&=\begin{bmatrix}
1 & 0& 0& 0\\
0 &0.8 &0.8 &0\\
0& 0.2& 0.2& 0\\
0 &0 & 0& 1
\end{bmatrix}
\cdot  [\partial_1 v_1, \partial_1 v_2, \partial_2v_1,\partial_2v_2]^T.
\end{align}
That is, we would have 
\be
QD'Q = \begin{bmatrix}
1 & 0& 0& 0\\
0 &0.8 &0.8 &0\\
0& 0.2& 0.2& 0\\
0 &0 & 0& 1
\end{bmatrix}
\text{ or } 
D' = \begin{bmatrix}
0 & 0& 0& 0.3\\
0 &0.5 &0 &0\\
0& 0& 0.5& 0\\
0 &0 & 0& 0.5
\end{bmatrix},
\ee
which are not diagonal matrices. Hence, this example shows that our model indeed covers more operators that those discussed in \cite{2018arXiv180201895B}.

\subsection{Numerical simulations and observations}\label{sec:num}
Let $\A$ be the operator defined in Subsection \ref{uatta_sec}, and let
\be
\Sigma[\Aep]:= \flp{\B_{s,t}:\,\,s,t\in[0,1]}
\ee
where, for $0\leq s,t\leq 1$, $\B_{s,t}$ are the first order operators introduced in \eqref{skewed_symmetric_gradient}. As we remarked before,  the seminorm $PGV^2_{\B_{s,t}}$ interpolates between the $TGV^2$ and $NsTGV^2$ regularizers. We define the {cost function} $\CC(\alpha, s,t)$ to be
\be\label{cost_fun}
\CC(\alpha, s,t):=\norm{u_{\alpha,\B_{s,t}}-u_c}_{L^2(Q)}.
\ee
From Theorem \ref{main_thm} we have that $\CC(\alpha, s,t)$ admits at least one minimizer $(\ta,\tilde s,\tilde t)\in \R^+\times[0,1]\times[0,1]$.\\\\
To explore the numerical landscapes of the cost function $\mathcal C(\alpha,s,t)$, we consider the discrete box-constraint 
\begin{multline}\label{parameter_domain}
(\alpha_0,\alpha_1, s,t)\in \flp{0.025,\,0.05,\,0.075,\ldots, 1}\\
\times \flp{0.025,\,0.05,\,0.075,\ldots, 1}\times\flp{0,\,0.025,\,0.05,\,\ldots,\,1}\times\flp{0,\,0.025,\,0.05,\,\ldots,\,1}.
\end{multline}
We perform numerical simulations of the images shown in Figure \ref{fig:clean_noise}: the first image represents a clean image $u_c$, whereas the second one is a noised version $u_\eta$, with heavy artificial Gaussian noise. The reconstructed image $u_{\alpha,\B}$ in Level 2 of our training scheme is computed by using the primal-dual algorithm presented in \cite{chambolle2011first}.

\begin{figure}[!h]
  \centering
        \includegraphics[width=1.0\linewidth]{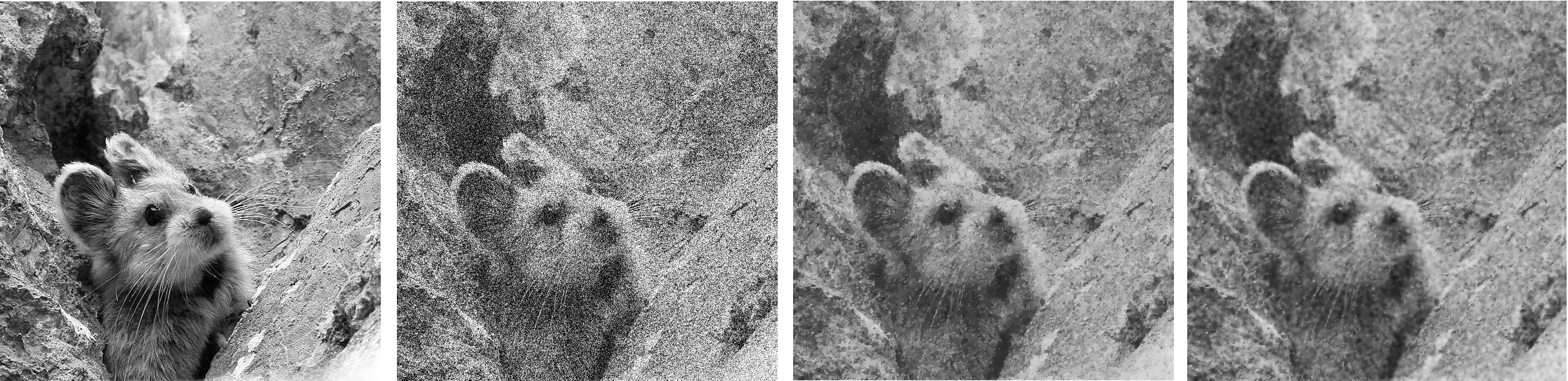}
\caption{From left to right: the test image of a Pika; a noised version (with heavy artificial Gaussian noise); the optimally reconstructed image with $TGV$ regularizer; the optimally reconstructed image with $PGV$ regularizer.}
\label{fig:clean_noise}
\end{figure}

It turns out that the minimum value of \eqref{cost_fun}, taking values in \eqref{parameter_domain}, is achieved at $\ta_0=5.6$, $\ta_1=1.2$, $\tilde s=0.8$, and $\tilde t=0.2$. The optimal reconstruction $u_{\ta,\B_{\tilde s,\tilde t}}$ is the last image in Figure \ref{fig:clean_noise}, whereas the optimal result with $\B_{s,t}\equiv\mathcal E$, i.e., $u_{\ta,TGV}$, is the third image in Figure \ref{fig:clean_noise}. Although the optimal reconstructed image $u_{\ta,\B_{\tilde s,\tilde t}}$ and $u_{\ta,\mathcal E}$ do not present too many differences to the naked eye, we do have that
\be
\CC(\ta,\tilde s,\tilde t)<\CC(\ta,0.5,0.5)
\ee
( see also Table \ref{table_test_result} below). That is, the reconstructed image $u_{\ta,\B_{\tilde s,\tilde t}}$ is indeed ``better" in the sense of our training scheme ($L^2$-difference). 

\begin{table}[!h]
\begin{tabular}{|l|l|l|l|l|l|}
\hline
Regularizer & optimal solution & minimum cost value   \\ \hline
$TGV^2$ & $\tilde \alpha_0=0.074$, $\tilde \alpha_1=0.625$  & $\CC(\ta,0.5,0.5)=18.653$    \\ \hline
$PGV^2$&$\tilde \alpha_{0}=0.072$, $\tilde \alpha_1=0.575$, $\tilde s=0.95$, $\tilde t=0.05$ & $\CC(\ta,\tilde s,\tilde t)= 17.6478 $  \\ \hline
\end{tabular}
\caption{minimum cost value with different regularizers. The minimum value of the cost function for the $PGV^2$- regularizer is approximately $5\%$ below that of the $TGV^2$- regularizer.}\label{table_test_result}
\end{table}

\vspace{-.5cm}

To visualize the change of cost function produced by different values of $(s,t)\in [0,1]^2$, we fix $\bar \alpha_0=0.072$ and $\bar \alpha_1=0.575$ and plot in Figure \ref{fig:mesh_contour_0_1_skew} the mesh and contour plot of $\mathcal C(\bar \alpha,s,t)$.

\begin{figure}[!h]
\begin{subfigure}{.495\textwidth}
  \centering
        \includegraphics[width=1.0\linewidth]{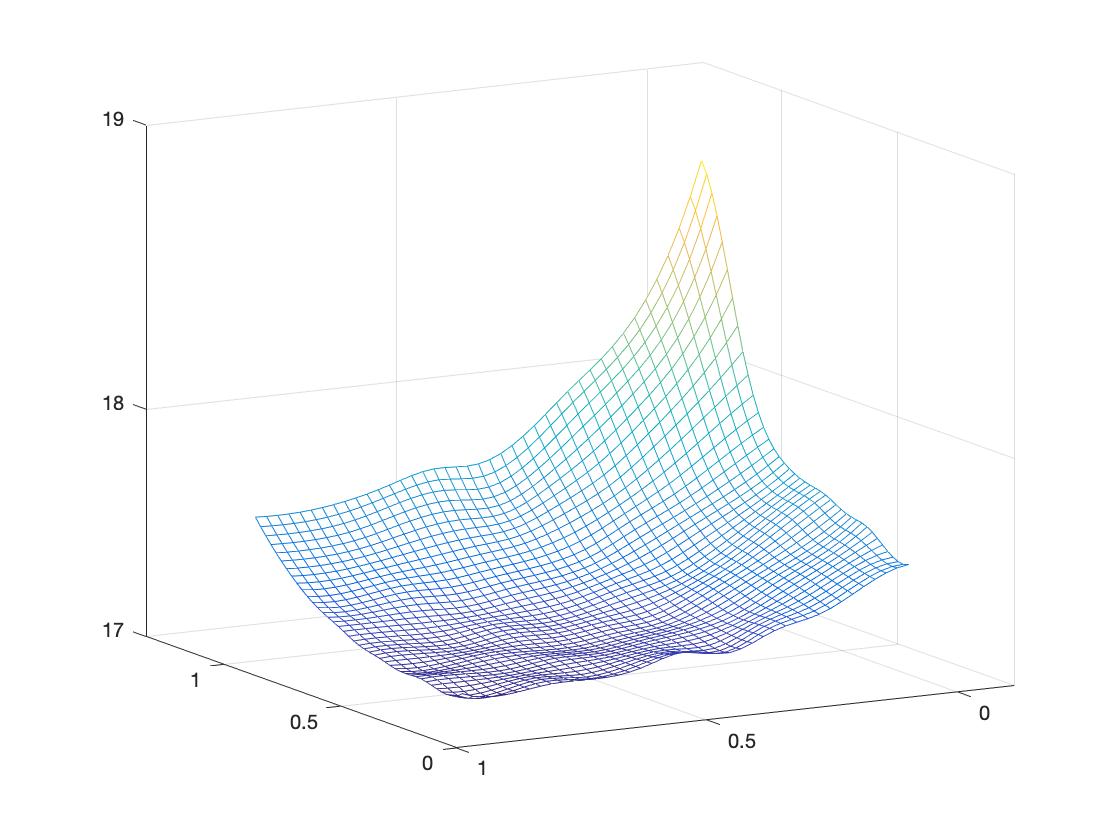}
\end{subfigure}
\begin{subfigure}{.495\textwidth}
  \centering
        \includegraphics[width=1.0\linewidth]{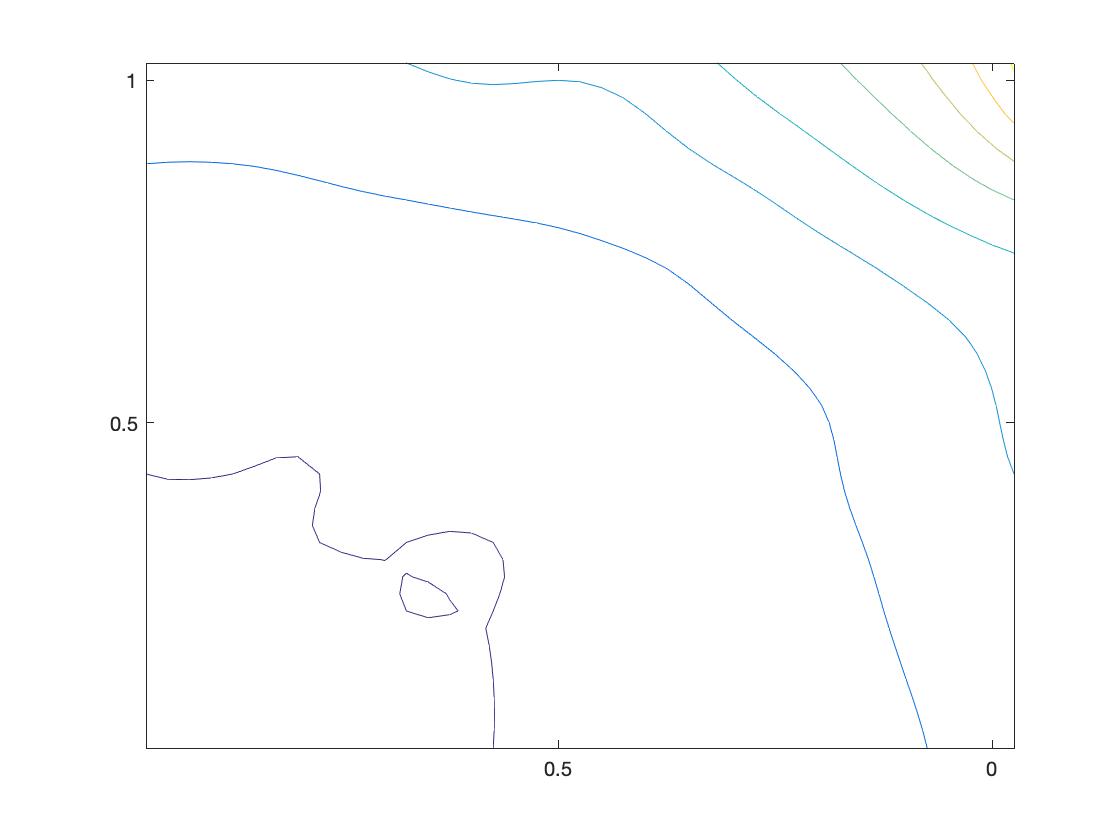}
\end{subfigure}
\caption{From the left to the right: mesh and contour plot of the cost function $\mathcal C(\bar \alpha,s,t)$ in which $\bar \alpha=(\bar \alpha_0,\bar\alpha_1)$ is fixed, $(s,t)\in[0,1]^2$.}
\label{fig:mesh_contour_0_1_skew}
\end{figure}
We again remark that the introduction of $PGV_{\alpha,\B[k]}$ regularizers into the training scheme is only meant to expand the training choices, but not to provide a superior seminorm with respect to the popular choices $TGV^2$ or $NsTGV^2$. The fact whether the optimal regularizer is $TGV^2$, $NsTGV^2$ or an intermediate regularizer is completely dependent on the given training image $u_\eta=u_c+\eta$.

\section*{Acknowledgements}
The work of Elisa Davoli has been funded by the Austrian Science Fund (FWF) project F65 ``Taming complexity in partial differential systems". Irene Fonseca thanks the Center for Nonlinear Analysis for its support during the preparation of the manuscript. She was supported by the National Science Foundation under Grand No. DMS-1411646. The work of Pan Liu has been supported by the Centre of Mathematical Imaging and Healthcare and funded by the Grant ''EPSRC Centre for Mathematical and Statistical Analysis of Multimodal Clinical Imaging" with No. EP/N014588/1. All authors are thankful to the Erwin Schr\"odinger Institute in Vienna, where part of this work has been developed during the workshop ``New trends in the variational modeling of failure phenomena".

\end{document}